\documentclass{amsart}
\usepackage[utf8]{inputenc}

\usepackage{fullpage}

	\usepackage{amsmath}
 \usepackage{amsfonts}
	\usepackage{amssymb}
	\usepackage{amsthm}
	\usepackage{braket}
	\usepackage{xcolor}
\usepackage[shortlabels]{enumitem}
\usepackage{stmaryrd}
	
 \usepackage{float}
	\usepackage{graphicx}
	\usepackage{hyperref}
    \usepackage{cleveref}
	\usepackage{ifsym}
	\usepackage{indentfirst}
	\usepackage{mathrsfs}
	\usepackage{mathtools}
	\usepackage{proof}
	\usepackage{qtree}
	\usepackage{setspace}
	\usepackage{tensor}
 \usepackage{tabularx}
	\usepackage{tikz}
	\usepackage{tikz-cd}
	\usepackage{tabstackengine}
        \usepackage{cancel}

\makeatletter
\newsavebox{\@abr}
\newcommand{\llangle}[1][]{\savebox{\@abr}{\(\m@th{#1\langle}\)}%
  \mathopen{\copy\@abr\mkern2mu\kern-0.9\wd\@abr\usebox{\@abr}}}
\newcommand{\rrangle}[1][]{\savebox{\@abr}{\(\m@th{#1\rangle}\)}%
  \mathclose{\copy\@abr\mkern2mu\kern-0.9\wd\@abr\usebox{\@abr}}}

  \newsavebox{\@sbr}
\newcommand{\llsquare}[1][]{\savebox{\@sbr}{\(\m@th{#1{[}}\)}%
  \mathopen{\copy\@sbr\mkern3mu\kern-0.9\wd\@sbr\usebox{\@sbr}}}
\newcommand{\rrsquare}[1][]{\savebox{\@sbr}{\(\m@th{#1{]}}\)}%
  \mathclose{\copy\@sbr\mkern3mu\kern-0.9\wd\@sbr\usebox{\@sbr}}}
\makeatother

	\providecommand{\corollaryname}{Corollary}
	\providecommand{\definitionname}{Definition}
	\providecommand{\examplename}{Example}
	\providecommand{\lemmaname}{Lemma}
	\providecommand{\propositionname}{Proposition}
	\providecommand{\remarkname}{Remark}
	\providecommand{\theoremname}{Theorem}
	\providecommand{\setupname}{Setup}
	\providecommand{\conjecturename}{Conjecture}
	\providecommand{\questionname}{Question}
	\providecommand{\objectivename}{Objective}
	\providecommand{\aimname}{Aim}
	\providecommand{\notationname}{Notation}

	\theoremstyle{plain}
		\newtheorem{thm}{\protect\theoremname}[section] 
		\newtheorem{prop}[thm]{\protect\propositionname}
		\newtheorem{lem}[thm]{\protect\lemmaname}
		\newtheorem{cor}[thm]{\protect\corollaryname}

	\theoremstyle{definition}
		\newtheorem{defn}[thm]{\protect\definitionname}
		\newtheorem{example}[thm]{\protect\examplename}
		\newtheorem{setup}[thm]{\protect\setupname}

		\newtheorem{notn}[thm]{\protect\notationname}

	\theoremstyle{remark}
		\newtheorem{rem}[thm]{\protect\remarkname}
		
	\numberwithin{figure}{section}
	\numberwithin{equation}{section}

	\usetikzlibrary{matrix,arrows,decorations.pathmorphing,positioning,decorations.pathreplacing}
	\tikzset{commutative diagrams/.cd, 
		mysymbol/.style = {start anchor=center, end anchor = center, draw = none}}
    \tikzset{
    labl/.style={anchor=north, rotate=90, inner sep=2mm}
    }

	\tikzcdset{every label/.append style = {font =  }}
	\tikzcdset{scale cd/.style={every label/.append style={scale=#1},
    cells={nodes={scale=#1}}}}

		\newcommand{\rad}{\operatorname{rad}\nolimits}

		\newcommand{\im}{\operatorname{im}\nolimits}

	\newcommand{\maxideal}{\mathfrak{m}}
 	\newcommand{\jacrad}{\mathcal{A}}

\newenvironment{acknowledgements}{
		\begin{abstract}} {\end{abstract}}

	\makeatletter
\DeclareRobustCommand{\rvdots}{%
  \vbox{
    \baselineskip4\p@\lineskiplimit\z@
    \kern-\p@
    \hbox{.}\hbox{.}\hbox{.}
  }}

  \newcommand{\compactBZhat}{\skew{1}\widehat{\rule{0ex}{1.4ex}\smash{\mathbb{Z}}}}
    \newcommand{\compactBQhat}{\skew{1}\widehat{\rule{0ex}{1.4ex}\smash{\mathbb{Q}}}}

\newcommand\tikznode[2]{\tikz[remember picture,baseline=(#1.base)]{\node(#1)[inner sep=0pt]{#2};}}

\newcommand{\overleftharpoon}[2]{\overset{{}_{\leftharpoonup}}{#1}_{\hspace{-0.2mm}#2}}
\newcommand{\overrightharpoon}[2]{\overset{{}_{\rightharpoonup}}{#1}_{\hspace{-0.2mm}#2}}

\makeatother

\title{String algebras over local rings: 
regular examples}
\author{Raphael Bennett-Tennenhaus}
\address{Raphael Bennett-Tennenhaus\newline
		Department of Mathematics,\newline 
		Aarhus University, 
		Ny Munkegade 118,  
		8000 Aarhus C,  
		Denmark}
\email{raphaelbennetttennenhaus@gmail.com}
\date{}


\begin{document}

\begin{abstract} 
String algebras, in the usual sense, are finite-dimensional algebras over a given ground field.
We recall a generalisation of the definition of a string algebra, which was introduced in a previous paper of the author. 
This generalisation replaces the ground field with a noetherian local ring.
We provide a way to generate examples of string algebras over any regular local ring, which depends on the Krull dimension. 
\end{abstract}

\maketitle

\section{Introduction}
\label{section-intro}

Butler and Ringel \cite{ButRin1987} defined string algebras as $k$-algebras of the form  $kQ/I$, where $kQ$ is the path algebra of a quiver $Q$ over a ground field $k$, with $I$ being admissible, and such that additional combinatorial conditions hold. 
Previously, Wald and Waschb\"{u}sch \cite{WalWas1985} leveraged the module classification of string algebras from \cite{ButRin1987} to derive a corresponding classification for the special-biserial algebras introduced by Skowro\'{n}ski  and Waschb\"{u}sch \cite{SkoWas1983}. 
Assem and Skowro\'{n}ski  \cite{AssSko1986} introduced gentle algbras, a specific class of string algebras known for the more manageable behavior of their derived category. 
String algebras have been the testing ground for a wide range of instruments from representation theory and other related topics in algebra; see for example  
\cite{Baur-Coelho-Simoes-A-geometric-model-for-the-module-category-of-a-gentle-algebra, BekMer2003, BruDouMouThoYil2020, CanSch2021, CanPauSch2021, ChaGuaSua2023, FraGirRiz2021, GirRueVel2022, GupKubSar2022, LabFragSchVal2022, Kalck-Singularity-categories-of-gentle-algebras, PalPilPla2021, Schroer-Modules-without-self-extensions-over-gentle-algebras}. 
In \cite{benn-tenn-string-alg-over-loc} we introduced the notion of a string algebra over a noetherian local ground ring. 
In this article we give a way of generating examples. 

\subsection{Admissible ideals}

For a noetherian local ring $(R,\maxideal,k)$ we write  $\maxideal Q$ for the kernel of the canonical  $R$-algebra map $RQ\to kQ$ and we let $A=\bigoplus_{\ell\geq 1}A_{\ell}$ where  $A_{\ell}\subseteq RQ$ is the $R$-span of the length-$\ell$ paths. 
We say that  an ideal  $I$ of $RQ$ is \emph{admissible} provided there are integers $m,n>0$ such that
\[
\begin{array}{cccc}
     I\subseteq A+ \maxideal Q,
     & 
     I\cap A \subseteq A^{2}+ \maxideal Q\cap A,
     &
     A^{m}\subseteq I+ \maxideal Q,
     & 
     \maxideal Q\subseteq I+\sum_{\ell=1}^{n} A_{\ell}.
\end{array}
\]
On the one hand, any semiperfect, module-finite and elementary $R$-algebra is the quotient $RQ/I$ of a path algebra by an admissible ideal; see \cite[Theorem 1.1]{benn-tenn-string-alg-over-loc}. 
On the other hand, any such quotient is semiperfect, module-finite and elementary - and furthermore, $Q$ is the Gabriel quiver of $RQ/I$ when $R$ is $\maxideal$-adically complete; see \cite[Theorem 1.2]{benn-tenn-string-alg-over-loc}. 
Admissible ideals were studied by Raggi-C\'{a}rdenas and Salmer\'{o}n \cite{CarSal1987}. 

\subsection{String algebras over local rings. }
We say that quiver $Q$ is \emph{biserial} provided each vertex is the head (respectively, tail) of at most $2$ distinct arrows.
We say that an ideal $I$ of $RQ$ is \emph{special} if, for any arrow $b$, there is at most $1$ arrow $a$ such that $t(a)=h(b)$ and $ab\notin I$ (respectively, at most $1$ arrow $c$ with $h(c)=t(b)$ and $bc\notin I$). 
We say $I$ is \emph{arrow}-\emph{direct} if for any arrow $a$ we have
\[
\begin{array}{cc}
    RQ a \cap \sum_{x} RQ x\subseteq I,
    &
    aRQ  \cap \sum_{y} yRQ \subseteq I,
\end{array}
\]
where sums run through the arrows $x$ (respectively, $y$)  such that $t(x)=t(a)$ (respectively, $h(y)=h(a)$). 

By a \emph{string algebra} over a local ring $R$ we mean an $R$-algebra of the form $RQ/I$ where $Q$ is a biserial quiver and $I$ is an ideal in $RQ$ that is admissible, special and arrow-direct; see \cite[Definition 1.4]{benn-tenn-string-alg-over-loc}. 

Since $Q$ is finite, our notion of a  string algebra over a field coincides with that of a  \emph{string algebra} as in Butler and Ringel  \cite{ButRin1987}; see \cite[Proposition 6.2]{benn-tenn-string-alg-over-loc}.  
\Cref{thm-examples-of-string-algebras} below provides a way of generating  examples of string algebras over regular local rings. 
 We fix some terms and notation.  
\begin{itemize}
    \item A \emph{biserial} quiver is one where each vertex is the head (respectively tail) of at most two arrows. 
    \item A set $Z$ of paths in a quiver $Q$ defines a \emph{special pair} $(Q,Z)$ if for any arrow $b$ there is at most one arrow $a$  with $t(a)=h(b)$ and $ab\notin Z$ (respectively, at most one arrow $c$ with $h(c)=t(b)$ and $bc\notin Z$).

    For the remaining terms and notation below we suppose $(Q,Z)$ is a special pair. 
    \item A $Z$-\emph{inadmissible} path has the form $pzq$ for paths $p$, $q$ and $z\in Z$, and otherwise it is $Z$-\emph{admissible}.  
     \item A non-trivial cycle $c$ is $Z$-\emph{primitive} provided each of its powers $c^{n}$ is $Z$-admissible, and provided $c$ itself cannot be written as a non-trivial power of another $Z$-admissible cycle of shorter length. 
     \item A  non-trivial path $p=a_{n}\dots a_{1}$ with arrows $a_{i}$ defines a set of vertices \emph{traversed by} $p$, namely
     \[
     V(p)=\{h(a_{n}),t(a_{1})\}\cup\{ h(a_{i})=t(a_{i+1})\colon 1\leq i<n\}.
     \] 
    \item A \emph{primitive} vertex is one in  the union  $V=\bigcup_{c} V(c)$ where $c$ runs through the $Z$-primitive cycles. 
    The corresponding \emph{primitive}-\emph{nerve} partition $V=V[1]\sqcup \dots\sqcup V[n]$ is defined by  transitively closing the reflexive and symmetric relation $\bigcup_{c}V(c)\times V(c) 
    $ on $V$, and partitioning   by equivalence classes.  
\end{itemize} 

\begin{example}
\label{example-introduction}
    Let $Q$ be the quiver, and $Z$ be the set of relations, defined by
    \[
\begin{array}{ccc}
Q\quad=\quad
\begin{tikzcd}
1\arrow[out=100,in=160,loop,  "", "a"',distance=0.6cm]\arrow[r,"x", swap]
&
2\arrow[l,"y", swap, bend right]\arrow[r,"w", swap]
&
3
\arrow[l,"z", swap, bend right]
&
4
\arrow[l,"f", swap]\arrow[out=120,in=60,loop, "b",distance=0.45cm]
&
5
\arrow[l,"g", swap]\arrow[out=80,in=20,loop, "c",distance=0.6cm]
\end{tikzcd}
&
\quad
&
Z=\quad\{
a^{2},\,
xy,\,
wx,\,
yz,\,
zf,\,
fb,\,
bg,\,
gc,\,
c^{\ell}\}
\end{array}
\]
\end{example}
where $\ell>0$ is some fixed integer. 
The $Z$-primitive cycles are $ayx$, $yxa$, $xay$, $wz$, $zw$, and $b$
and we have
\[
\begin{array}{ccc}
V(ayx)=V(yxa)=V(xay)=\{1,2\},
&
V(wz)=V(zw)=\{2,3\},
&
V(b)=\{4\} 
. 
\end{array}
\]
Hence 
$V =\{1,2\}\cup\{2,3\}\cup \{4\}$, and $\bullet-\bullet\quad\bullet$ depicts the graph defined by the nerve of the subsets above in $V$, since $\{1,2\}\cap\{2,3\}\neq \emptyset$ and $\{1,2\}\cap\{4\}=\emptyset=\{2,3\}\cap \{4\}$. Hence  $V[1]=\{1,2,3\}$ and $V[2]=\{4\}$.


\begin{thm}
\label{thm-examples-of-string-algebras}
Let $Z$ be a set of paths of length at least $2$ such that $(Q,Z)$ is special, $V=V[1]\sqcup \dots\sqcup V[n]$ be the primitive-nerve partition,  $(R,\maxideal,k)$ be a regular local ring of dimension $n$, $( s_{1},\dots,s_{n}) $ be a regular sequence of generators of $\maxideal$, and define $\sigma_{v}\in RQ$ by the sum of the $Z$-primitive cycles at  $v\in V$. 
Let 
 \[
\begin{array}{cc}
 I = \langle Z\rangle 
+
\sum_{i=1}^{n}
\langle s_{i}e_{v}-\sigma_{v}\,\mid\, v\in V[i]\rangle 
+
\langle s_{i}e_{v}\,\mid\, v\in Q_{0}\setminus V[i]\rangle. 
& 
(*)
\end{array}
 \]
Any $Z$-admissible path lies outside $I$, and if $Q$ is biserial then $RQ/I$ is a string algebra over $R$.
\end{thm}

\begin{example}
\label{exampple-introduction-again}
Let $\compactBZhat_{p}$ be the ring of $p$-adic integers and consider the Iwasawa algebra $R=\compactBZhat_{p}[\vert t\vert ]$, which is a regular local ring of Krull dimension $2$. 
Here $\maxideal=\langle p,t\rangle$ is the maximal ideal of $R$, and the residue field is $k=R/\maxideal\cong\mathbb{F}_{p}$. 
Let $s_{1}=p$ and $s_{2}=t$. 
For this choice of $R$ we apply \Cref{thm-examples-of-string-algebras} in the context of \Cref{example-introduction}, which can be done since the number of parts in the primitive-nerve partition is $2$. 
Here 
 \[
\begin{array}{c}
 I=\langle 
Z\rangle 
+\langle 
pe_{1}-ayx-yxa,\,
pe_{2}-xay-zw,\,
pe_{3}-wz,\,
te_{4}-b
\rangle
+\langle
te_{1},\,
te_{2},\,
te_{3},\,
te_{5},\,
pe_{4},\,
pe_{5}
\rangle.
\end{array}
 \]
In \Cref{example-intro-example-revisited} we describe the ring $\Lambda=RQ/I$. This a string algebra over $R$ by \Cref{thm-examples-of-string-algebras}. 
\end{example}

\begin{prop}
    \label{cor-rickes-examples}
    Let $\Gamma=\overline{KQ}\,/\,\overline{\langle Z\rangle}$ be a completed string algebra in the sense defined by Ricke \cite[p.~28]{Ric2017}. Then  $\Gamma\cong RQ/I$ where $R=K[\vert t_{1},\dots t_{n}\vert ]$ and where $I$ has the form $(*)$ from \Cref{thm-examples-of-string-algebras}. 
\end{prop}

The paper is structured as follows. 
In \S\ref{sec-primitive-cyces-and-nerve-partitions} we build background results needed for \S\ref{subsec-prim-reg-rad-collections}, where primitive-nerve partitions are discussed and analysed. 
In \S\ref{sec-string-algs} we prove \Cref{thm-examples-of-string-algebras} and \Cref{cor-rickes-examples}. 
In \S\ref{sec-examples} we use \Cref{thm-examples-of-string-algebras} to  produce examples of string algebras over mixed characteristic local rings.

\section{Zero-relations and primitive cycles}
\label{sec-primitive-cyces-and-nerve-partitions}


We will use language commonly used when dealing with path algebras of quivers, such as \emph{non}-\emph{trivial paths}. 
We also use terminology which is not widely used, such as \emph{arrow}-\emph{distinct} and \emph{initial subpaths}. 
Although they are somewhat self-explanatory, we specify explicitly our  meaning of all of these phrases below. 
This is done all at once for brevity in the sequel, and the author intends the lists below to work as a dictionary.
\begin{notn}
    Notation used for \emph{quivers} $Q=(Q_{0},Q_{1},h,t)$ make up the next four items.
\begin{itemize}
    \item $Q_{0}$ is the set of vertices and $Q_{1}$ is the set of arrows. 
    \item $h\colon Q_{1}\to Q_{0}$ is the head function, sending any arrow $a$ to its head $h(a)$. 
    \item $t\colon Q_{1}\to Q_{0}$ is the tail function, sending any arrow $a$ to its tail $t(a)$. 
\end{itemize}
\end{notn}

In  \S\ref{sec-primitive-cyces-and-nerve-partitions} we always let $Q=(Q_{0},Q_{1},h,t)$ be a quiver whose sets $Q_{0}$ and $Q_{1}$ of vertices and edges are finite. 

\begin{defn}
\label{defn-all-the-words}
    Terminology used for 
\emph{paths} make up the next six items.
\begin{itemize}
    \item \emph{Trivial} (\emph{length}-$0$)  paths $e_{v}$  are defined for each vertex $v$. We let $h(e_{v})=v=t(e_{v})$.
    \item \emph{Non}-\emph{trivial paths of length} $n\geq 1$, denoted  $a_{n}\dots a_{1}$, are defined for any \emph{consecutive} collection of arrows $a_{i}$ $(i=1,\dots,n)$,  meaning $t(a_{i+1})=h(a_{i})$ for $1\leq i<n$. 
    \item The \emph{head} and \emph{tail} of a non-trivial path $a_{n}\dots a_{1}$ are $h(a_{n})$ and $t(a_{1})$ respectively. 
    \item \emph{Composing} paths $p$ of length $n$ and $q$ of length $m$ with $t(p)=h(q)$, denoted $pq$, is defined by $pq=p$ if $m=0$, $pq=q$ if $n=0$, and arrow concatenation  if $m,n>0$. 
    \item A path \emph{traverses} a vertex $v$ if said path has the form $pq$ where $t(p)=v=h(q)$. 
    \item A path $a_{n}\dots a_{1}$ is said to  \emph{have distinct arrows} provided  $a_{i}\neq a_{j}$ for $i\neq j$. 
    \end{itemize}
    
    Terminology used for 
\emph{subpaths} make up the next five items.
    \begin{itemize}
    \item \emph{Subpaths} of a path $p$ are the paths $z$ satisfying $p=qzr$ for some paths $q,r$. 
    \item \emph{Right} subpaths of a path $p$ are the paths $z$ satisfying $p=qz$ for some path $q$. 
    The \emph{right arrow} of a non-trivial path $p$ is the unique length-$1$ right subpath of $p$. 
    \item \emph{Left} subpaths of a path $p$ are the paths $z$ satisfying $p=zr$ for some path $r$. 
    The \emph{left arrow} of a non-trivial path $p$ is the unique terminal length-$1$ subpath of $p$. 
\end{itemize}

Terminology and notation concerning \emph{cycles} make up the next three items.
\begin{itemize}
    \item \emph{Cycles}  are non-trivial paths with the same head and tail.
    \item The \emph{incidence} of a cycle $p$ is the vertex $v(p)$ which satisfies $h(p)=v(p)=t(p)$. 
    \item A \emph{rotation} of a cycle $c=c_{n}\dots c_{1}$ refers to $c$ or to $ c_{i}\dots c_{1}c_{n}\dots c_{i+1}~(1\leq i<n)$.
    \item A \emph{non}-\emph{trivial power} of a cycle $c$ refers to $c^{k}=c\dots c$ for some $k> 1$. A \emph{power} of $c$ refers either to $c$ or to a non-trivial power of $c$. 
\end{itemize}

Terminology concerning a fixed set $Z$ of paths in $Q$ make up the next two items. 
\begin{itemize}
    \item A paths $p$ is $Z$-\emph{admissible}  if no element of $Z$ arises as a subpath of $p$, denoted $Z\nmid p$. Otherwise  some $z\in Z$ is a subpath $p$, and $p$ \emph{factors through} $Z$, denoted $Z\mid p$. 
    \item A cycle $c$ is $Z$-\emph{primitive} if $c$ is not a power of shorter cycle and each power of $c$ is  $Z$-admissible. 
\end{itemize}
\end{defn}

 
%

\begin{example}
    \label{example-running-zeroth-non-trivial}
 Let $Q$ be the quiver, and let $Z$ be the set of zero-relations, defined by
    \[
\begin{array}{c}
Q=\quad
\begin{tikzcd}
1\arrow[out=100,in=160,loop,  "", "x"',distance=0.6cm]
\arrow[r,swap,"b_{1}"]
&
2\arrow[l,"a_{5}", swap, bend right]
&
3\arrow[out=120,in=60,loop, "y",distance=0.45cm]
\arrow[l,"a_{4}", swap]
&
4\arrow[l,"a_{3}", swap]
\arrow[r,swap,"b_{2}"]
&
5\arrow[l,"a_{2}", swap, bend right]
\arrow[r,swap,"b_{3}"]
&
6\arrow[out=80,in=20,loop, "z",distance=0.6cm]
\arrow[l,"a_{1}", swap, bend right]
\end{tikzcd}\vspace{2mm}
\\
Z=\quad\{a_{i+1}a_{i}\mid 1\leq i\leq 4\}\cup \{b_{1}x,\,xa_{5},\,a_{4}y,\,y^{4},\,ya_{3},\,b_{3}b_{2},\,a_{1}b_{3},\,z^{2}\}.
\end{array}
\]
Here the $Z$-primitive cycles are  $x$, $a_{5}b_{1}$, $b_{1}a_{5}$, $b_{2}a_{2}$, $a_{2}b_{2}$,  $zb_{3}a_{1}$,  $a_{1}zb_{3}$  and $b_{3}a_{1}z$. 
\end{example}

\subsection{Special zero-relations}

We begin by understanding some of the terminology given in \Cref{defn-all-the-words}. 

\begin{lem}
\label{lem-cycles-with-admissible-powers-are-powers-of-primitive-ones}
If $c$ is a cycle and every power of $c$ is  $Z$-admissible, then $c$ is a power of a $Z$-primitive cycle. 
\end{lem}

\begin{proof}
There is nothing to prove if $c$ is $Z$-primitive. 
So suppose $c$ is not $Z$-primitive. 
By assumption $c=d^{k}$ for some cycle $d$, necessarily  $Z$-admissible, where $k>1$. 
The claim now follows by a straightforward induction on the length of $c$. 
To see this, note that the base case is immediate: if $c$ is a loop that is  $Z$-admissible then it is $Z$-primitive. For the induction step note that $d$ is shorter than $c$, and any power of $d$ is a subpath of a power of $c$ and so  $Z$-admissible. 
\end{proof}

\begin{lem}
\label{lem-arrow-distinct-primitive-iff-rotation-also}
Any rotation of a $Z$-primitive cycle is  $Z$-primitive. 
\end{lem}

\begin{proof}

Let $c$  be a $Z$-primitive cycle   and let $d$ be a rotation of $c$. 
For any integer $m>0$ the path $d^{m}$ is a subpath of $c^{m+1}$ and hence $Z\nmid c^{m+1}$ implies $Z\nmid d^{m}$. 
We require that $d$ is not a power of a shorter cycle. For a contradiction suppose $d=b^{l}$ for some integer $l>1$ and some cycle $b=b_{k}\dots b_{1}$ of length $k<n$. 
Hence $kl=n$. 
Now write  $c=c_{n}\dots c_{1}$ and  $d=c_{i}\dots c_{1}c_{n}\dots c_{i+1}$ and $i=kt+r$ for some integers $t,r$ with $0\leq r<k$. 
Note that if $r=0$ then $c=b^{l}$ which is gives the required contradiction since $c$ is $Z$-primitive. 
So we assume $r>0$. 
In this notation we have $c_{i}\dots c_{1}=b^{t}b_{k}\dots b_{k-r+1}$ and $c_{n}\dots c_{i+1}=b_{k-r}\dots b_{1}b^{l-t-1}$, 
which means $c$ is the product $a^{l}$ of the rotation $a=b_{k-r}\dots b_{1}b_{k}\dots  b_{k-r+1}$ of $b$. 
\end{proof}

\begin{defn}
    \label{defn-spec-quiver-and-relations}
    We say that $(Q,Z)$ is \emph{special} if the following conditions hold. 
\begin{itemize}
     \item[(SP1)] For any arrow $b$ there exists at most $1$ arrow $a$ with $t(a)=h(b)$ and  $ab\notin Z$.
    \item[(SP2)] For any arrow $b$ there exists at most $1$ arrow $c$ with $h(c)=t(b)$ and  $bc\notin Z$.
\end{itemize}
\end{defn}

\begin{rem}
    \label{remark-special-pair-unique-paths}
    Let $(Q,Z)$ be special. Suppose that  $p$ and $p'$ are  $Z$-admissible paths such that $p'$ is not longer than $p$ and such that $p$ and $p'$ have the same right (respectively, left) arrow. 
    Using the same argument from the proof of of \cite[Lemma 5.4]{benn-tenn-string-alg-over-loc}, one can show that $p'$ is a right (respectively, left) subpath of $p$. 
\end{rem}


We now combine our results to understand $Z$-primitive cycles when $(Q,Z)$ is special.

\begin{lem}
\label{lem-primitive-cycles-arrow-unique}
If $(Q,Z)$ is special then any $Z$-primitive cycle has distinct arrows.
\end{lem}

\begin{proof}
Let $c=c_{n}\dots c_{1}$ be a $Z$-primitive cycle. 
Of course, any loop has distinct arrows, so we assume $n>1$. 
For a contradiction suppose $c$ does not have distinct arrows. 
By \Cref{lem-arrow-distinct-primitive-iff-rotation-also} the rotations of $c$ must all be $Z$-primitive, but they must not have distinct arrows. 
Hence after rotation we can assume $c_{i+1}=c_{1}$ for an integer $i$ with $1\leq i < n$, and we choose $i$ to be minimal with this property. 
For any integer $s$ where $si\leq n$ we have that the paths $c_{si}\dots c_{1} $ and $(c_{i}\dots c_{1} )^{s}$ are  $Z$-admissible and that they have the same right arrow and the same length. 
Hence by \Cref{remark-special-pair-unique-paths} they are equal, and so $c_{si}=c_{i}$. 

Choose $t>0$ maximal such that $ti\leq n$. 
Hence $(t+1)i\neq n$, and $(t+1)i>n$ since $c$ is $Z$-primitive.  
Let $p=c_{1}c_{n}\dots c_{ti}$ and $p'=c_{n-ti+1}c_{n-ti} \dots c_{1}c_{i}$. 
Both these paths have right arrow $c_{ti}=c_{i}$ and length $n-ti+2$, so they are equal by \Cref{remark-special-pair-unique-paths}. 
But now $c_{n-ti+1}=c_{1}$, contradicting the minimality of $i>n-ti$.
\end{proof}

\begin{lem}
\label{lem-primitive-cycles-sharing-arrows} 
If $(Q,Z)$ is special then any $Z$-primitive cycles that share an arrow must be rotations of each other. 
In particular, there cannot be distinct $Z$-primitive cycles with the same right (respectively, left) arrow. 
\end{lem}

\begin{proof}
We are considering $Z$-primitive cycles $c=c_{n}\dots c_{1}$ and $d=d_{m}\dots d_{1}$, and assuming that that $c_{i}=d_{j}$ for some $i$ and $j$. 
We complete the proof by showing that $n=m$ and that $c$ is a rotation of $d$. 
Thus we relabel rotations of $c$ and $d$ by letting
\[
\begin{array}{cc}
c'=c_{i}\dots c_{1}c_{n}\dots c_{i+1}, 
&
d'=d_{j}\dots d_{1}d_{m}\dots d_{j+1}.
\end{array}
\]
Hence $c'$ and $d'$ are, respectively, rotations of $c$ and $d$ which have the same left arrow. 
Without loss of generality, and for a contradiction, assume $n< m$. By \Cref{remark-special-pair-unique-paths} this means  $d'=c'd''$ where $d''$ is some   $Z$-admissible cycle at $v(c')=v(d')$. 
Since $c'$ is $Z$-primitive the path $c'c'$ is  $Z$-admissible, and has the same left arrow as $d'$ meaning that one of $c'c'$ and $d'$ is a  left subpath of the other by \Cref{remark-special-pair-unique-paths}. In particular, $c'$ and $d''$ have the same left arrow, and so $d'$ does not have distinct arrows, and so $d$ does not have distinct arrows. 
We now have the required contradiction, since $d$ does not have distinct arrows, but $d$ is $Z$-primitive, which is impossible by \Cref{lem-primitive-cycles-arrow-unique}. 
\end{proof}

\begin{cor}
\label{cor-bound-on-number-of-primitive-cycles}
Let $(Q,Z)$ be special and $u$ be a vertex that is the head of $n$ arrows and the tail of $m$ arrows. 
The number of $Z$-primitive cycles at $u$ is at most $\min \{n,m\}$.  
\end{cor}

\begin{proof}
Without loss of generality $n\leq m$. 
For a contradiction suppose there are $n+1$ distinct $Z$-primitive cycles at $u$. 
Then there must be two of these cycles with the same left arrow, contradicting \Cref{lem-primitive-cycles-sharing-arrows}. 
\end{proof}

\subsection{Ideals generated by zero-relations}
\label{subsection-path-algebras-over-local-rings}
\begin{setup}
\label{setup-set-Z-of-paths}
\label{setup-R-m-k}
In  \Cref{subsection-path-algebras-over-local-rings} we let $(R,\maxideal,k)$ be a noetherian local ring,  
 $RQ$ be the path algebra and $\langle Z\rangle$ be the ideal in $RQ$ generated by a set $Z$ of paths where $(Q,Z)$ is special. 
 For any vertex $v$ let $\sigma_{v}=\sum c$, the sum of the $Z$-primitive cycles $c\in Q_{Z}^{\,\prime}(v\circlearrowleft)$ at $v$. 
\end{setup}


\begin{notn}
    Let $u$ be a vertex and let $n\geq0$ be an integer. 
    \begin{itemize}
    \item $Q^{n}_{Z}(u\leftarrow)$ denotes the  (set of) $Z$-admissible paths $p$ of length-$n$ with  $h(p)=u$. 
    \item $Q^{n}_{Z}(\leftarrow u)$ denotes the  $Z$-admissible  paths $p$  of length-$n$ with  $t(p)=u$. 
    \item $Q_{Z}(v\leftarrow )$ denotes  $\bigcup_{n\geq 0}Q^{n}_{Z}(v\leftarrow )$, the  $Z$-admissible paths with head $v$.
    \item $Q_{Z}(\leftarrow v)$ denotes  $\bigcup_{n\geq 0}Q^{n}_{Z}(\leftarrow v)$, the  $Z$-admissible paths with tail $v$. 
    \item $Q_{Z}(u\leftarrow v)$ denotes 
 $Q_{Z}(\leftarrow v)\cap Q_{Z}(u\leftarrow ) $, the  $Z$-admissible paths from $v$ to $u$.
    \item $Q_{Z}(u\circlearrowleft)$ denotes $Q_{Z}(u\leftarrow u)$, the  $Z$-admissible cycles $c$  with $v(c)=u$.
    \item $Q^{\,\prime}_{Z}(u\circlearrowleft)$ denotes the subset of $Q_{Z}(u\circlearrowleft)$ consisting of $Z$-primitive cycles. 
\end{itemize}
\end{notn}

\begin{lem}
\label{lem-special-paths-of-arbitrary-length-initial-subpaths-or-primitive-cycles}
Let $u$ be a vertex such that $Q_{Z}^{h}(\leftarrow u)\neq \emptyset$ (respectively, $Q_{Z}^{h}(u\leftarrow )\neq \emptyset$) for infinitely many $h>0$. 
Then there exists $d>0$ such that, for any $\ell>d$, if a path $p$ lies  in $Q_{Z}^{\ell}(\leftarrow u)$ (respectively, $ Q_{Z}^{\ell}(u\leftarrow )$) then $p$ is a right (respectively, left) subpath of a power of some $c\in Q^{\,\prime}_{Z}(u\circlearrowleft)$. 
\end{lem}

\begin{proof}

Since $Q$ is finite there is a bound $g$ on the length of the paths with distinct arrows and tail $u$. 
Define $S\subseteq Q^{1}_{Z}(\leftarrow u)$ as the set of arrows $a$ which are not the right arrow of a primitive cycle at $u$. 
Momentarily fix $a\in S$ and an integer $m>g$ and consider  a path $p\in Q_{Z}^{m}(\leftarrow u)$ with right arrow $a$. 
By construction $p$ cannot have distinct arrows, and so $p$ is a right subpath of a power of some  $Z$-admissible cycle $c$  with right arrow $a$. 

If $Z\nmid c^{h}$ for all $h>0$ then,  by \Cref{lem-cycles-with-admissible-powers-are-powers-of-primitive-ones}, $c$ is a power of a $Z$-primitive cycle with  right arrow $a$. 
Since this is impossible, any $a\in S$ is the right arrow of some  cycle $c_{a}$ with $Z\mid c_{a}^{h}$ for $h\gg 0$. 

Now for each $a\in S$ choose an integer $d_{a}>0$ such that $Z\mid c_{a}^{h}$ whenever the length $\ell$ of $c_{a}^{h}$ satisfies  $\ell>d_{a}$. 
Now let $d=g+\sum_{a\in S}d_{a}>0$. 
By construction, if $\ell>d$ and $p\in Q^{\ell}_{Z}(\leftarrow u)$ then $p$ is a right subpath of a power of a cycle $c$ such that $Z\nmid c^{h}$ for all $h>0$. 
By \Cref{lem-cycles-with-admissible-powers-are-powers-of-primitive-ones} this means any such $p$ is a right subpath of a power of a $Z$-primitive cycle.
The respective statement follows by symmetry. 
\end{proof}


\begin{lem}
\label{lem-prim-cycles-swapping-with-arrows}
Let $u $ be a vertex and $a$ be an arrow,  and let $c\in Q_{Z}^{\,\prime}(u\circlearrowleft)$. 
   \begin{enumerate}
    \item If $a$ is neither the first arrow $c_{1}$ of $c$, nor last arrow $c_{n}$ of $c$, then $ac,ca\in \langle Z\rangle$. 
    \item If $a=c_{n}$ then there exists $c'\in Q^{\,\prime}_{Z}(t(a)\circlearrowleft) $ unique such that $ca=ac'$ in $RQ$.
    \item If $a=c_{1}$ then there exists $c'\in Q^{\,\prime}_{Z}(h(a)\circlearrowleft) $ unique such that $ac=c'a$ in $RQ$.
   \end{enumerate} 
\end{lem}

\begin{proof}
We let $c=c_{n}\dots c_{1}$  for arrows $c_{i}$.  
Note that $Z\nmid cc_{n}$ and $Z\nmid c_{1}c$.

(1) Since $a\neq c_{n}$, if $h(a)=u$ then by  (SP2) we have $Z\mid c_{1}a$    and so $Z\mid ca$. 
If $h(a)\neq u$ then $ca=0$ in $RQ$. 
In either case we have $ca\in \langle Z\rangle$, and by symmetry $ac\in \langle Z\rangle$.

(2) If $n=1$ then $ca=a^{2}=ac$. 
So let $n>1$ and $ca=c_{n}c_{n-1}\dots c_{1}c_{n}=ac_{n-1}\dots c_{1}c_{n}$. 
Let $c'=c_{n-1}\dots c_{1}c_{n}$, which is $Z$-primitive by \Cref{lem-arrow-distinct-primitive-iff-rotation-also}. 
Furthermore, if $ac''=ca$ for any  $Z$-admissible cycle $c''$ with $v(c'')=u$, then by \Cref{remark-special-pair-unique-paths} we have that $c'=c''$ since $ac'=ac''$. 
The proof of (3) is similar.
\end{proof}

\begin{lem}
\label{lem-CB-sum-of-prim-cycles}\emph{(c.f. \cite[Lemma 3.1]{Cra2018})}
For any path $p$ we have $p\sigma_{t(p)}-\sigma_{h(p)}p\in \langle Z\rangle$. 
\end{lem}

\begin{proof}
It suffices to assume $p$ is not a trivial path.  We can assume $p$ is  $Z$-admissible. 
We prove the assertion by induction on the length of $p$. 
Suppose $p=a$ is an arrow. 
By \Cref{lem-primitive-cycles-sharing-arrows} there is at most one $Z$-primitive cycle $c$ with first arrow $a$, and at most one $Z$-primitive cycle $c'$ with last arrow $a$. 
By \Cref{lem-prim-cycles-swapping-with-arrows} such a cycle $c$ exists if and only if such a cycle $c'$ exists, in which case $ac=c'a$. 

If $d\neq c$ is a $Z$-primitive cycle at $v(d)=t(a)$, then by \Cref{lem-primitive-cycles-sharing-arrows} the first arrow of $d$ cannot be $a$, which by \Cref{lem-prim-cycles-swapping-with-arrows} means $ad\in \langle Z\rangle$. 
Likewise if $d'\neq c'$ is a  $Z$-primitive cycle at $v(d')=h(a)$ then $d'a\in \langle Z\rangle$. 
Assuming without loss of generality that one and hence both of $c,c'$ exist, altogether we have
\[
a\sigma_{t(a)}-\sigma_{h(a)}a=a(\sigma_{t(a)}-c)-(\sigma_{h(a)}-c')a\in  \langle Z\rangle.
\]
Now suppose $p$ has length $n>1$, and write $p=qa$ for some path $q$ of length $n-1$. 
Then
\[
\begin{array}{c}
p\sigma_{t(p)}-
\sigma_{h(p)}p
=
qa\sigma_{t(a)}-
\sigma_{h(p)}p
=
q(a\sigma_{t(a)}-
\sigma_{h(a)}a)+
q\sigma_{h(a)}a-
\sigma_{h(q)}qa
\\
= q(a\sigma_{t(a)}-
\sigma_{h(a)}a)+
(q\sigma_{t(q)}-
\sigma_{h(q)}q)a
\in 
(q\sigma_{t(q)}-
\sigma_{h(q)}q)a+
\langle Z\rangle.
\end{array}
\]
So, if the assertion holds for $q$ it does so for $p$, as required. 
\end{proof}


In \Cref{lem-one-sided-form-of-coeff-prim-relations} we show how to factor $Z$-admissible paths pass over $Z$-primitive cycles. 

\begin{cor}
    \label{lem-one-sided-form-of-coeff-prim-relations}
For any $s\in R$, vertex $v$ and paths $p\in Q_{Z}(\leftarrow v)$ and $q\in Q_{Z}(v\leftarrow)$ we have 
\[
\begin{array}{c}
(se_{h(p)}-\sigma_{h(p)})pq+ \langle Z\rangle
=
p(se_{v}-\sigma_{v})q+ \langle Z\rangle
=pq(se_{t(q)}-\sigma_{t(q)})+ \langle Z\rangle.
\end{array}
\]
\end{cor}

\begin{proof}
By \Cref{lem-CB-sum-of-prim-cycles} we have  $\sigma_{h(p)}p-p\sigma_{t(p)}\in \langle Z\rangle$. 
That the left-hand equation holds follows by writing $p(se_{v}-\sigma_{t(p)})$ as the sum of $sp-\sigma_{h(p)}p$ and $\sigma_{h(p)}p-p\sigma_{v}$. 
The right-hand equation is similar. 
\end{proof}


\begin{notn}
\label{lem-factoring-through-prim-cycle-bijection}
   Let $u,v,w$ be vertices, $q\in Q_{Z}(u\leftarrow w)$, $m\in Q_{Z}(\leftarrow u)$ and $n\in Q_{Z}(w\leftarrow )$. 
\begin{itemize}
    \item $P(q;v)$ is the set of $(p',p)\in Q_{Z}(u\leftarrow v)\times  Q_{Z}(v\leftarrow w)$ such that $q=p'p~(v\in Q_{0})$. 
    \item $V(q)$ is the set of vertices $v$ traversed by $q$, that is, the set of $v$ such that $P(q;v)\neq\emptyset$.  
    \item $Q^{+}_{Z}(m\mid\leftarrow  w)$ is the set of paths in $Q_{Z}( \leftarrow w)$ of the form $mp$ with $p\in Q_{Z}(u\leftarrow w)$. 
    \item $Q^{+}_{Z}(u\leftarrow\mid  n)$ is the set of paths in $Q_{Z}( \leftarrow w)$ of the form $p'n$ with $p'\in Q_{Z}(u\leftarrow w)$. 
    \item $Q^{-}_{Z}(m\mid\leftarrow  w)$ is the set of paths $p\in Q_{Z}(u \leftarrow w)$ such that $mp\in Q_{Z}(u\leftarrow w)$. 
    \item $Q^{-}_{Z}(u\leftarrow\mid  n)$ is the set of paths $p'\in Q_{Z}(u \leftarrow w)$ such that $p'n\in Q_{Z}(u\leftarrow w)$. 
\end{itemize}
Hence the notation above describes paths that  factor  through other paths. 
Establishing more notation for the sequel, note that there are bijections between these sets, defined by 
\[
\begin{array}{ccc}
Q^{+}_{Z}(m\mid\leftarrow  w)\leftrightarrow Q^{-}_{Z}(m\mid\leftarrow  w) & & Q^{+}_{Z}(u\leftarrow\mid  n)\leftrightarrow Q^{-}_{Z}(u\leftarrow\mid  n)
\\
mp\mapsfrom p,\quad  mp=q\mapsto m^{-1}q\coloneqq p&& p'n\mapsfrom p',\quad  p'n=q'\mapsto q'n^{-1}\coloneqq p'
\end{array}
\]
To see this, firstly note that $Q^{-}_{Z}(u\leftarrow\mid  n)\to Q^{+}_{Z}(u\leftarrow\mid  n)$ given by $p'\mapsto p'n$ is surjective by construction. 
Secondly, by \Cref{remark-special-pair-unique-paths} this function is also injective, since any two  $Z$-admissible paths with the same right arrow and equal length must coincide. 
\end{notn}

\begin{example}
    \label{example-running-first-non-trivial} In \Cref{example-running-zeroth-non-trivial},   $V(x)=\{1\}$,  $V(a_{5}b_{1})=\{1,2\}$, $V(b_{2}a_{2})=\{4,5\}$, and  $V(b_{3}a_{1}z)=\{5,6\}$. 
\end{example}

\begin{lem}
\label{lem-factoring-through-prim-cycles-V-sets}
Let $u$ and $w$ be vertices  and let $q\in Q_{Z}(u\leftarrow w)$. 
\begin{enumerate}
    \item There exists at most one $Z$-primitive cycle $c\in Q_{Z}^{\,\prime}(u\circlearrowleft)$ such that $q\in Q^{\pm }_{Z}(c\mid\leftarrow  w)$, and in case $q\in Q^{+ }_{Z}(c\mid\leftarrow  w)$ we have $V(c^{-1}q)\subseteq V(q)$ and $V(c)=V(q)$. 
    \item There exists at most one $Z$-primitive cycle $d\in Q_{Z}^{\,\prime}(w\circlearrowleft)$ such that  $q\in Q^{\pm }_{Z}(u\leftarrow\mid  d)$, and in case $q\in Q^{+ }_{Z}(u\leftarrow\mid  d)$ we have   $V(qd^{-1})\subseteq V(q)$ and $V(d)=V(q)$. 
\end{enumerate}
\end{lem}

\begin{proof}
(1)  
If $q\in Q^{-}_{Z}(c\mid\leftarrow  w)\cap  Q^{-}_{Z}(c'\mid\leftarrow  w)$ for $c,c'\in Q_{Z}^{\,\prime}(u\circlearrowleft)$ then $Z\nmid cq,c'q$ meaning  $c$ and $c'$ have the same right arrow by (SP1), and so $c=c'$ by \Cref{lem-primitive-cycles-sharing-arrows}. 
Similarly if $q\in Q^{+}_{Z}(c\mid\leftarrow  w)\cap Q^{+}_{Z}(c'\mid\leftarrow  w)$ then $c$, $q$ and $c'$ have the same left arrow, again giving $c=c'$ by \Cref{lem-primitive-cycles-sharing-arrows}.

Now assume $q\in Q^{+}_{Z}(c\mid\leftarrow  w)$, say where $q=cp$ for some $p\in Q_{Z}(u\leftarrow w)$. 
By \Cref{lem-factoring-through-prim-cycle-bijection} we have  $V(c)\subseteq V(q)$ and $V(p)\subseteq V(q)$. 
Furthermore, there is some $n\gg0$ such that $c^{n}$ and $q$ are  $Z$-admissible with the same left arrow, where $c^{n}$ is longer. 
Hence $q$ is a left subpath of $c^{n}$ by \Cref{remark-special-pair-unique-paths}, and so $V(q)\subseteq V(c)$. 

This shows (1) holds. The proof of (2) is similar and omitted.
\end{proof}

\begin{lem}
\label{lem-technical-for-K-coefficients}
Let $u,w$ be vertices and  $\sum_{p}r_{p}p\in RQ$ where  $p$ runs through the paths. 
Then 
\[
\sum_{q\in Q_{Z}(u\leftarrow w)}
(
r_{q}\sigma_{u} 
- 
\overleftharpoon{r}{q}e_{u}
)q
\in \langle Z\rangle \,
\text{where } \overleftharpoon{r}{q}=\begin{cases}
    r_{c^{-1}q} & (\exists c\in Q_{Z}^{\,\prime}(u\circlearrowleft)\colon q\in Q^{+}_{Z}(c\mid\leftarrow w)), 
    \\
    0 & (\text{otherwise}),
\end{cases}
\]
and
\[
\sum_{q\in Q_{Z}(u\leftarrow w)}
q(
r_{q}\sigma_{w} 
- 
\overrightharpoon{r}{q}e_{w}
)
\in \langle Z\rangle \,
\text{where } \overrightharpoon{r}{q}=\begin{cases}
    r_{qd^{-1}} & (\exists d\in Q_{Z}^{\,\prime}(w\circlearrowleft)\colon q\in Q^{+}_{Z}(u\leftarrow \mid d)), 
    \\
    0 & (\text{otherwise}).
\end{cases}
\]
\end{lem}

\begin{proof}
We only prove the first containment. The second is similar. 
Let $p\in Q_{Z}(u\leftarrow w)$. 
By \Cref{lem-factoring-through-prim-cycle-bijection}, if $c \in Q_{Z}^{\,\prime}(u\circlearrowleft)$ and $p'\in  Q_{Z}(u\leftarrow w)\setminus Q^{-}_{Z}(c\mid\leftarrow  w)$ then 
$Z\mid cp'$ and so $\sigma_{u}p-\sum_{c}cp\in \langle Z\rangle$ where $c$ runs through the $Z$-primitive cycles at $u$ with $p\in Q^{-}_{Z}(c\mid\leftarrow  w)$. 

For each $p$ let $c_{p}=\{c\in Q_{Z}^{\,\prime}(u\circlearrowleft)\colon p\in Q^{-}_{Z}(c\mid\leftarrow  w)\}$. By \Cref{lem-factoring-through-prim-cycles-V-sets} either $c_{p}$ is empty or is a singleton. 
The union $\bigcup_{p\in Q_{Z}(u\leftarrow w)}c_{p}$ is in bijection with $\bigcup_{c\in Q_{Z}^{\,\prime}(u\circlearrowleft)}Q^{-}_{Z}(c\mid\leftarrow  w)$. 
As in  \Cref{lem-factoring-through-prim-cycle-bijection} there is a bijection $Q^{-}_{Z}(c\mid\leftarrow  w)\to Q^{+}_{Z}(c\mid\leftarrow  w)$ sending $p\to cp$ whose inverse sends $q\in Q^{+}_{Z}(c\mid\leftarrow  w)$ to some $c^{-1}q\in Q^{-}_{Z}(c\mid\leftarrow  w)$. 
For each $q\in Q_{Z}(u\leftarrow w)$ we let $s_{q}=\sum_{c}r_{c^{-1}q}$ where $c$ runs through the elements of $c\in Q_{Z}^{\,\prime}(u\circlearrowleft) $ such that $q\in Q^{+}_{Z}(c\mid\leftarrow  w)$. 
Again by \Cref{lem-factoring-through-prim-cycles-V-sets} there is at most one such $c$. 
Hence $s_{q}=\overleftharpoon{r}{q}$.

Thus, combining this notation, we have
\[
\begin{array}{c}
\sum_{p\in Q_{Z}(u\leftarrow w)}r_{p}\sum_{c\in Q_{Z}^{\,\prime}(u\circlearrowleft)\colon p\in Q^{-}_{Z}(c\mid\leftarrow  w)}cp=
\sum_{c\in Q_{Z}^{\,\prime}(u\circlearrowleft)}\sum_{p\in Q^{-}_{Z}(c\mid\leftarrow  w)}r_{p}cp
\\
=
\sum_{c\in Q_{Z}^{\,\prime}(u\circlearrowleft)}\sum_{cp\in Q^{+}_{Z}(c\mid\leftarrow  w)}r_{p}cp
=
\sum_{c\in Q_{Z}^{\,\prime}(u\circlearrowleft)}\sum_{q\in Q^{+}_{Z}(c\mid\leftarrow  w)}r_{c^{-1}q}q
\\
=
\sum_{q\in Q_{Z}(u\leftarrow w)}\sum_{c\in Q_{Z}^{\,\prime}(u\circlearrowleft)\colon q\in Q^{+}_{Z}(c\mid\leftarrow  w)}r_{c^{-1}q}q
=
\sum_{q\in Q_{Z}(u\leftarrow w)}\overleftharpoon{r}{q}q. 
\end{array}
\]
Combining everything we have so far gives the first containment. 
The second containment is dual.
\end{proof}

\section{Primitive-nerve partitions}
\label{subsec-prim-reg-rad-collections}


\begin{defn}
        \label{defn-canonical-partition-vertices}
Let $(Q,Z)$ be special. 
Let $V=\bigcup _{c}V(c)$, the set of vertices $v$ which are traversed by some $Z$-primitive cycle $c$. 
We define a partition $V=V[1]\sqcup \dots\sqcup V[n]$ by taking connected components of the nerve associated to the set $V(c)$. 
We explain what this means. 
Consider the subset of $V\times V$ given by the union $\bigcup_{c} V(c)\times V(c)$. 
Define the equivalence relation $\sim$ on $V$ generated by (taking the transitive closure of) this  relation. 
Thus we have a partition, which we denote $V=V[1]\sqcup \dots\sqcup V[n]$, into $\sim$ equivalence classes. 

We refer to $V=V[1]\sqcup\dots \sqcup V[n]$ as the \emph{primitive}-\emph{nerve partition} of $(Q,Z)$.  
\end{defn}

\begin{setup}
\label{setup-scalars}
    In  \S\ref{subsec-prim-reg-rad-collections} we recall and add to \Cref{setup-set-Z-of-paths}.  
    Hence we begin by assuming $(R,\maxideal,k)$ is local and noetherian, that $Z$ is a set of paths, and that $(Q,Z)$ is special. 
    Additionally, we fix a collection $(s_{1},\dots,s_{n})$ of scalars $s_{i}\in R$, one for each part of the nerve partition $V=V[1]\sqcup\dots \sqcup V[n]$. 
\end{setup}

\begin{example}
\label{example-running-second-non-trivial} For the pair $(Q,Z)$ from \Cref{example-running-first-non-trivial} we have  
$V=\{1,2,4,5,6\}=Q_{0}\setminus \{3\}$, and the primitive-nerve partition is 
$V=V[1]\sqcup V[2]$ where $
V[1]=\{1,2\}$ and $V[2]=\{4,5,6\}$. 
\end{example}

\begin{rem}
\label{rem-equiv-rel-prim-cycles}
Let $u$ and $v$ be vertices.  
If $u\sim v$ and $u\neq v$ then there exists $n>1$ and $w_{1},\dots,w_{n}\in Q_{0}$  with $u=w_{1}$, $v=w_{n}$ and, for each $i<n$, $w_{i},w_{i+1}\in V(c_{(i)})$ for some primitive cycle $c_{(i)}$. 
As a consequence the following statements hold. 
To see (3) and (4), apply \Cref{lem-factoring-through-prim-cycles-V-sets}.
\begin{enumerate}
    \item If $u,v\in V$ and $u\sim v$, $d\in Q_{Z}^{\,\prime}(u\circlearrowleft)$ and $d'\in Q_{Z}^{\,\prime}(v\circlearrowleft)$ then there exist $Z$-primitive cycles $c_{(1)},\dots c_{(t)}$ where  $d=c_{(1)}$, $d'=c_{(n)}$ and $w_{i+1}\in V(c_{(i)})\cap V(c_{(i+1)})$ for $i<t$. 
    \item If $u,v\in V[l]$ then $V(d)\cup V(d')\subseteq V[l]$ for any  $d\in Q_{Z}^{\,\prime}(u\circlearrowleft)$ and $d'\in Q_{Z}^{\,\prime}(v\circlearrowleft)$.
    \item If $q\in Q^{+}_{Z}(c\mid\leftarrow  v)$ for some $c\in Q_{Z}^{\,\prime}(u\circlearrowleft)$ then $V(q)\subseteq V[i]$ for some $i=1,\dots,n$. 
    \item If $q'\in Q^{+}_{Z}(u\leftarrow\mid  d)$ for some $d\in Q_{Z}^{\,\prime}(v\circlearrowleft)$ then $V(q)\subseteq V[j]$ for some $j=1,\dots,n$. 
\end{enumerate}
\end{rem}

    


\begin{notn}
\label{notation-ideal-I}
     Let $I=I_{Z}+I_{V}+I_{\neg V}$ where the ideals $I_{Z},I_{V},I_{\neg V}$ of $RQ$ are defined by
\[
\begin{array}{ccc}
I_{V}=\sum_{i=1}^{n}\langle s_{i}e_{v}-\sigma_{v},\mid\, v\in V_{i}\rangle,
&
I_{\neg V}=\sum_{i=1}^{n}\langle s_{i}e_{v},\mid\, v\in Q_{0}\setminus V_{i}\rangle,
&
I_{Z}=\langle Z\rangle. 
\end{array}
\]
\end{notn}


\begin{lem}
\label{lem-canonical-form-in-J}
Let $u$ and $w$ be vertices and  $x\in  e_{u}I_{V}e_{w}$. There exist $r_{j,q,v}\in R$ for each $j=1,\dots,n$, $q\in Q_{Z}(u\leftarrow w)$ and $v\in V(q)$ such that $x-\sum_{q}
t_{q}q\in I_{Z} $  where
\[
t_{q}=
\begin{cases}
\sum_{v\in V(q)}s_{i}r_{i,q,v}
-
\sum_{v\in V(c^{-1}q)}r_{i,c^{-1}q,v}  & \left(\exists\, c\in Q_{Z}^{\,\prime}(u\circlearrowleft):q\in Q^{+}_{Z}(c\mid\leftarrow  w),u\in V[i]
\right)
\\
\sum_{j=1}^{d}
\sum_{v\in V[j]\cap V(q)}
s_{j}r_{j,q,v} & \left(\text{otherwise}\right).
\end{cases}
\]
\end{lem}

\begin{proof}
If $v$ is a vertex and $q'',q'$ are paths with $v\neq t(q'')$ or $v\neq h(q')$ then $q''(s_{v}e_{v}-\sigma_{v})q'=0$. 
So, elements of $e_{u}I_{V}e_{w}$ have the form
\[
\begin{array}{c}
\sum_{j=1}^{d}
\sum_{v\in V[j]}
\sum_{q''\in Q_{Z}(u\leftarrow v)}
\sum_{q'\in Q_{Z}(v \leftarrow w)}
r''_{j,v,q''}
q''
(s_{l}e_{v}-\sigma_{v})
r'_{j,v,q'}
q'
\end{array}
\]
for some $r'_{j,v,q'},r''_{j,v,q''}\in R$. 
By \Cref{lem-one-sided-form-of-coeff-prim-relations}, since $x\in e_{u}(I_{Z}+I_{V})e_{w}$ 
we therefore have
\[
\begin{array}{c}
x-
\sum_{j=1}^{d}
(s_{j}e_{u}-\sigma_{u})
\sum_{v\in V[j]}
\sum_{q\in Q_{Z}(u\leftarrow w)}
\sum_{(q'',q')\in P(q;v)}
r''_{j,v,q''}
r'_{j,v,q'}
q
\in I_{Z}.
\end{array}
\]
Let $r_{j,q,v}=\sum_{(q'',q')}r_{j,v,q''}''r_{j,v,q'}'$ for each $j$, $q$ and $v$. 
Hence the containment above becomes
\[
\begin{array}{cc}
x-
\sum_{j=1}^{d}
\sum_{q\in Q_{Z}(u\leftarrow w)}
\sum_{v\in V[j]\cap V(q)}
(s_{j}e_{u}-\sigma_{u})
r_{j,v,q}
q\in I_{Z}
\end{array}
\]
since $P(q;v)=\emptyset$ whenever $v\notin V(q)$. 
Fix $p\in Q_{Z}(u\leftarrow w)$ where $p\in Q^{+}_{Z}(c\mid\leftarrow  w)$ for some $c\in Q_{Z}^{\,\prime}(u\circlearrowleft)$. 
Recall that such a $c$ is unique by \Cref{lem-primitive-cycles-sharing-arrows}. 
Furthermore $V(c^{-1}p)\subseteq V(p)= V(c)\subseteq  V[i]$ for some $i=1,\dots,n$ by \Cref{lem-factoring-through-prim-cycles-V-sets} and \Cref{rem-equiv-rel-prim-cycles}. 
We let $y_{i,p}=\sum_{v\in V(c^{-1}p)}r_{i,c^{-1}q,v}$ in this situation. 
Otherwise, where $j\neq i$ or where $q\notin Q^{+}_{Z}(c\mid\leftarrow  w)$ for all $c\in Q_{Z}^{\,\prime}(u\circlearrowleft)$, we let $y_{j,q}=0$. 
For any $j=1,\dots, n$ and any $q\in Q_{Z}(u\leftarrow w)$ let  $y'_{j,q}=\sum_{v\in V[j]\cap V(q)}s_{j}r_{j,q,v}$. 
In this notation, by the containment above and \Cref{lem-technical-for-K-coefficients}, 
\[
\begin{array}{c}
x-\sum_{q\in Q_{Z}(u\leftarrow w)}
\sum_{j=1}^{d}(y_{j,q}+y_{j,q}')q\in I_{Z}.
\end{array}
\]
The claim follows, since $\sum_{j=1}^{n}y'_{j,q}=\sum_{v\in  V(q)}
s_{i}
r_{i,q,v}$ for each $q$ with $V(q)\subseteq V[i]$.  
\end{proof}

\begin{lem}
\label{lem-canonical-form-in-K}
Let $u$ and $w$ be vertices and $y\in  e_{u}I_{\neg V}e_{w}$.  
There exist $r_{l,q,v}\in R$ for each $j=1,\dots,d$, $q\in Q_{Z}(u\leftarrow w)$ and $v\in V(q)$ such that $y-\sum_{q}
t_{q}q\in I_{Z} $  where
\[
\begin{array}{c}
t_{q}=
\begin{cases}
\sum_{j=1,\, j\neq i}^{d}
\sum_{v\in V(q)}
s_{j}r_{j,q,v} & \left(\exists\, c\in Q_{Z}^{\,\prime}(u\circlearrowleft) :q\in Q^{+}_{Z}(c\mid\leftarrow  w),u\in V[i]
\right)
\\
\sum_{j=1}^{d}
\left(
\sum_{v\in V(q)\setminus V[j]}
s_{j}r_{j,q,v}
\right)
& \left(\text{otherwise}\right).
\end{cases}
\end{array}
\]
\end{lem}

\begin{proof}
If $y\in e_{u}I_{\neg V}e_{w}$ then there exist elements $r'_{j,v,q'},r''_{j,v,q''}\in R$ such that 
\[
\begin{array}{c}
y-
\sum_{l=1}^{d}
\sum_{v\notin V[l]}
\sum_{q''\in Q_{Z}(u\leftarrow v)}
\sum_{q'\in Q_{Z}(v \leftarrow w)}
r''_{l,v,q''}
q''
s_{l}e_{v}
r'_{l,v,q'}
q'\in I_{Z}.
\end{array}
\]
Let $r_{l,q,v}=\sum_{(q'',q')}r_{j,v,q''}''r_{j,v,q'}'$ for each $j=1,\dots,n$, each $v\notin V[j]$ and each $q\in Q_{Z}(u\leftarrow w)$ where $(q'',q')$ runs through $P(q;v)$. 
Since $v\notin V(q)$ implies $P(q;v)=\emptyset $ this means
\[
\begin{array}{c}
y-
\sum_{j=1}^{d}
\sum_{q\in Q_{Z}(u\leftarrow w)}
\sum_{v\in V(q)\setminus V[j]}
r_{j,v,q}
s_{j}
q\in I_{Z}.
\end{array}
\]
As in the proof of \Cref{lem-canonical-form-in-J}, 
if  $p\in Q^{+}_{Z}(c\mid\leftarrow  w)$ for some $c\in Q_{Z}^{\,\prime}(u\circlearrowleft)$ then $V(c^{-1}p)\subseteq V(p)= V(c)\subseteq  V[i]$ for some $i=1,\dots,n$ by \Cref{lem-factoring-through-prim-cycles-V-sets} and \Cref{rem-equiv-rel-prim-cycles}, and so
\[
\begin{array}{c}
\sum_{j=1}^{d}
\sum_{v\in V(p)\setminus V[j]}
r_{j,v,p}
s_{j}
p
=
\sum_{j=1,\,j\neq i}^{d}
\sum_{v\in V(p)}
r_{j,v,p}
s_{j}
p,
\end{array}
\]
and so the claim follows by combining the equalities and containments above. 
\end{proof}


\begin{lem}
\label{lem-canonical-form-in-I}
Let $u$ and $w$ be vertices and 
 $z\in  e_{u}Ie_{w}$. There exist $r_{j,q}\in R$ for each $j=1,\dots,n$ and $q\in Q_{Z}(u\leftarrow w)$ such that $z-\sum_{q}
t_{q}q\in I_{Z} $  where
\[
\begin{array}{c}
t_{q}=
\begin{cases}
-
r_{i,c^{-1}q}+\sum_{j=1}^{n}
s_{j}r_{j,q}
  & \left(\exists\, c\in Q_{Z}^{\,\prime}(u\circlearrowleft) :q\in Q^{+}_{Z}(c\mid\leftarrow  w),u\in V[i]
\right)
\\
\sum_{j=1}^{n}
s_{j}r_{j,q} & \left(\text{otherwise}\right).
\end{cases}
\end{array}
\]
Conversley we have $\sum_{q}t_{q}q\in e_{u}(I_{Z}+I_{V}+I_{\neg V})e_{w}$ for any $r_{j,q}\in R$ defining $t_{q}\in R$ as above.
\end{lem}

\begin{proof}
By assumption we can choose $x\in I_{V}$ and $y\in I_{\neg V}$ such that $z-x-y\in I_{Z}$. 
By \Cref{lem-canonical-form-in-J} there exist elements  $r'_{j,q,v}\in R$   such that $x-\sum_{q}
t'_{q}q\in I_{Z}$  where
\[
\begin{array}{c}
t'_{q}=
\begin{cases}
\sum_{v\in V(q)}s_{i}r'_{i,q,v}
-
\sum_{v\in V(c^{-1}q)}r'_{i,c^{-1}q,v}  & \left(\exists\, c\in Q_{Z}^{\,\prime}(u\circlearrowleft) :q\in Q^{+}_{Z}(c\mid\leftarrow  w),u\in V[i]
\right),
\\
\sum_{j=1}^{n}
\sum_{v\in V[j]\cap V(q)}
s_{j}r'_{j,q,v} & \left(\text{otherwise}\right).
\end{cases}
\end{array}
\]
By \Cref{lem-canonical-form-in-K} there exist elements $r''_{j,q,v}\in R$  such that $y-\sum_{q}
t''_{q}q\in I_{Z} $  where
\[
\begin{array}{c}
t''_{q}=
\begin{cases}
\sum_{j=1,\, j\neq i}^{n}
\sum_{v\in V(q)}
s_{j}r''_{j,q,v} & \left(\exists\, c\in Q_{Z}^{\,\prime}(u\circlearrowleft)\, :\,q\in Q^{+}_{Z}(c\mid\leftarrow  w),u\in V[i]
\right),
\\
\sum_{j=1}^{n}
\sum_{v\in V(q)\setminus V[j]}
s_{j}r''_{j,q,v}
& \left(\text{otherwise}\right).
\end{cases}
\end{array}
\]
For each $j=1,\dots,d$, $q\in Q_{Z}(u\leftarrow w)$ and $v\in V(q)$ let $r'_{j,q,v}=0$ when $v\notin V[j]$ and let $r''_{j,q,v}=0$ when $v\in V[j]$. 
Hence when $q\in Q^{+}_{Z}(c\mid\leftarrow  w)$ for some $c\in Q_{Z}^{\,\prime}(u\circlearrowleft)$ we have $V(c^{-1}q)\subseteq V(q)\subseteq V[i]$ by Corollary \ref{rem-equiv-rel-prim-cycles}, and so $r''_{i,c^{-1}q,v}=0$ when $v\in V(c^{-1}q)$ and $r''_{i,q,v}=0$ when $v\in V(q)$. 
Combining the above, the first statement follows by letting $r_{j,q}=\sum_{v\in V(q)}r'_{j,q,v}+r''_{j,q,v}$ for each $j=1,\dots,n$ and each $q\in Q_{Z}(u\leftarrow w)$.
 
We now prove the second statement. 
For any $q\in Q_{Z}(u\leftarrow w)$, if $u\notin V[j]$ and $j=1,\dots, n$ then $s_{j}r_{j,q}q=s_{j}e_{u}r_{j,q}q\in e_{u} I_{\neg V}  e_{w}$. 
Hence if $u\notin V$ then  $t_{q}q\in e_{u}I_{\neg V} e_{w}$ for all $q$, so it suffices to assume $u\in V[i]$ for some $i=1,\dots, n$. 
As above $\sum_{j=1,\, i\neq j}^{n}s_{j}r_{j,q}q\in e_{u} I_{\neg V}  e_{w}$.  

By definition we have $s_{i}e_{u}-\sigma_{u}\in I_{V}$. 
Thus the second statement follows from \Cref{lem-technical-for-K-coefficients}, since 
\[
\begin{array}{c}
\sum_{q}\left(
(t_{q}-\sum_{j=1,\, i\neq j}^{n}s_{j}r_{j,q})e_{u}
+(s_{i}e_{u}-\sigma_{u})r_{i,q}
\right)
q=
-\sum_{q}
(
\sigma_{u}r_{i,q}
-
\overleftharpoon{r}{i,q}e_{u}
)
q
\end{array}
\]
where $\overleftharpoon{r}{i,q}=
    r_{i,c^{-1}q}$ when $q\in Q^{+}_{Z}(c\mid\leftarrow  w)$ for some $c\in Q_{Z}^{\,\prime}(u\circlearrowleft)$, and otherwise $\overleftharpoon{r}{i,q}=0$. 
\end{proof}

We note here that, although technical, \Cref{lem-canonical-form-in-I} was an important realisation for both the proof and the formulation of \Cref{thm-examples-of-string-algebras}.

\subsection{Regular sequences}
\label{subsec-regular}
The notion of a regular sequence comes from commutative algebra. 
Although well-known, to ensure the article is self contained we recall the definition  in \Cref{remark-regular-perms}. 

\begin{rem}
    \label{remark-regular-perms}
    Recall a sequence $(s_{1},\dots,s_{n})$  in $R$ is \emph{regular} if, for the ideals $\mathfrak{a}(0)=0$ and $\mathfrak{a}(i)=\langle s_{1},\dots, s_{i}\rangle$ in $R$
 for $i>0$, the coset $s_{i+1}+\mathfrak{a}(i)$ is a non-zero-divisor inside the ring $R/\mathfrak{a}(i)$  whenever $0\leq i<n$. 
 Since $(R,\maxideal,k)$ is noetherian and local, note that if $s_{1},\dots,s_{n}\in\maxideal$ and if $(s_{1},\dots, s_{n})$ is regular then for any permutation $\sigma\colon \{1,\dots,n\}\to\{1,\dots,n\}$ the sequence $(s_{\sigma (1)},\dots, s_{\sigma (n)})$ is also regular; see for example part $(\alpha)$ of the Remark, and then the following Corollary, in  \cite[pp. 126--127]{Matsumura-commutative-ring-theory}.  
\end{rem}

\begin{setup}
\label{setup-regular-scalars}
    In  \S\ref{subsec-regular} we recall and add to \Cref{setup-set-Z-of-paths} and \Cref{setup-scalars}. 
    From now on assume  that $s_{1},\dots,s_{n}\in \maxideal$ and that $(s_{1},\dots,s_{n})$ is a regular sequence in $R$. 
    Hence for any permutation $\sigma$ the sequence $(s_{\sigma (1)},\dots, s_{\sigma (n)})$ is also regular; see \Cref{remark-regular-perms}. 
\end{setup}

The proof of \Cref{thm-coeff-prim-relations-give-non-zero-admissible-paths} reveals how \Cref{lem-canonical-form-in-I} combines with the notion of a regular sequence. 

\begin{thm}
\label{thm-coeff-prim-relations-give-non-zero-admissible-paths}
    Let $x=\sum r_{q}q\in RQ$ where $q$ runs through $Q_{Z}(u\leftarrow w)$ for some vertices $u$ and $w$. 
    If $x\in e_{u}Ie_{w}$ then $r_{p}\in\maxideal$ for any path $p\in Q_{Z}(u\leftarrow w)$ of minimal length such that $r_{p}\neq 0$. 
\end{thm}

\begin{proof}
By \Cref{lem-canonical-form-in-I} we have $x-\sum_{q}
t_{q}q\in I_{Z} $  where  
\[
\begin{array}{c}
t_{q}=
\begin{cases}
-
r_{i,c^{-1}q} +\sum_{j=1}^{d}
s_{j}r_{j,q}
 & \left(\exists\, c\in Q_{Z}^{\,\prime}(u\circlearrowleft) :q\in Q^{+}_{Z}(c\mid\leftarrow  w),u\in V[i]
\right)
\\
\sum_{j=1}^{d}
s_{j}r_{j,q} & \left(\text{otherwise}\right).
\end{cases}
\end{array}
\]
for some  $r_{j,q}\in R$ where $j=1,\dots,n$ and $q\in Q_{Z}(u\leftarrow w)$. 
Since $x-\sum_{q}t_{q}q$ is an $R$-linear combination of $Z$-admissible paths that lies in $I_{Z}$, it is also an $R$-linear combination of $Z$-inadmissible paths. 
Since the paths form an $R$-basis of the path algebra $RQ$, we have that  $x=\sum_{q}t_{q}q$, and furthermore, that $t_{p}\notin \maxideal$ and $t_{q}=0$ for all $Z$-admissible paths $q\in Q_{Z}(u\leftarrow w)$ that are shorter than $p$. 
 For each  $j=1,\dots,n$ let $\mathfrak{n}_{j}=\langle s_{h}\colon j\neq h=1,\dots,n\rangle $, an ideal  of $R$. 
We  consider two special cases (a) and (b) below. 

(a) In case (a)  suppose $q\in Q_{Z}(u\leftarrow w)$ and $q\notin Q^{+}_{Z}(c\mid\leftarrow  w) $ for all $c\in Q_{Z}^{\,\prime}(u\circlearrowleft)$. This means $t_{q}=\sum_{j=1}^{d}
s_{j}r_{j,q}\in\maxideal$. Now additionally assume $q$ is not longer than $p$. 
Hence by assumption $t_{q}=0$.  
Combined, this means that for each $j=1,\dots,n$ we have $s_{j}r_{j,q}\in \mathfrak{n}_{j}$. 
Given that $(s_{1},\dots,s_{n})$ is a regular sequence, $s_{j}+\mathfrak{n}_{j}$ is regular in the quotient ring $R/\mathfrak{n}_{j}$, and altogether we have $r_{j,q}\in \mathfrak{n}_{j}$; see  \Cref{remark-regular-perms}. 

(b) In case (b) suppose $q\in Q_{Z}(u\leftarrow w)$ and $q\in Q^{+}_{Z}(c\mid\leftarrow  w) $ for some $c\in Q_{Z}^{\,\prime}(u\circlearrowleft)$. 
By \Cref{lem-primitive-cycles-sharing-arrows} this $c$ is unique, and we choose $m>0$ maximal such that $q=c^{m}q'$ for some $q'\in Q_{Z}(u\leftarrow w)$. 
If $q'\in Q^{+}_{Z}(c'\mid\leftarrow  w) $ for some $c'\in Q_{Z}^{\,\prime}(u\circlearrowleft)$ then $Z\nmid cc',c'c'$ which means $c$ and $c'$ have the same right arrow by (SP1), and so $c'=c$ by \Cref{lem-primitive-cycles-sharing-arrows}, contradicting the maximality of $m$. 
So $q'\notin Q^{+}_{Z}(c'\mid\leftarrow  w) $ for all $c'\in Q_{Z}^{\,\prime}(u\circlearrowleft)$.

By case (a) we can assume that $p=c^{m}p'$ for some $c\in Q_{Z}^{\,\prime}(u\circlearrowleft)$, for otherwise and $m>0$ maximal. 
Choose $i=1,\dots,n$ such that $v\in V[i]$. 
By case (b) we must have  $p'\notin Q^{+}_{Z}(c'\mid\leftarrow  w) $ for all $c'\in Q_{Z}^{\,\prime}(u\circlearrowleft)$. 
Let $p(0)=p'$ and $p(l)=c^{l}p'$ for each $l=1,\dots,m$. 
For any such $l$ we have $p(l-1)=c^{-1}p(l)$, and hence
\[
\begin{array}{cc}
t_{p(l)}
= 
-r_{i,p(l-1)}+s_{i}r_{i,p(l)}
+
\sum_{j=1,\, j\neq i}^{d}
s_{j}r_{j,p(l)}
\in s_{i}r_{i,p(l)}-r_{i,p(l-1)} + \mathfrak{n}_{i}.
&
(*)
\end{array}
\]
We claim that $r_{i,p(l-1)}\in\mathfrak{n}_{i}$ for each $l$.  
For $l=1$, since $p'=p(0)$ is shorter than $p$, by case (a) we have $r_{i,p(0)}\in\mathfrak{n}_{i}$, and so the claim holds. 
When $l<m$, assuming $r_{i,p(l-1)}\in\mathfrak{n}_{i}$ the inclusion $(*)$ above gives $s_{i}r_{i,p(l)}\in\mathfrak{n}_{i}$. 
Since the sequence $(s_{\sigma(1)},\dots,s_{\sigma(n)})$ is regular for each $\sigma$ we conclude $r_{j,p(l)}\in \mathfrak{n}_{i}$. 

Thus the claim holds by induction. Taking $l=m$ and recalling $p(m)=p$, the claim together with the inclusion $(*)$ above gives $t_{p}\in s_{i}r_{i,p}+\mathfrak{n}_{i}\subseteq \maxideal$, as required. 
\end{proof}

In \Cref{cor-coeff-prim-relations-give-non-zero-admissible-paths} we consider $x\in Q_{Z}(u\leftarrow w )$ in \Cref{thm-coeff-prim-relations-give-non-zero-admissible-paths}, and conclude that the $Z$-admissible paths define non-zero elements in the $R$-algebra $RQ/I$, providing a key step in the proof of \Cref{thm-examples-of-string-algebras}. 

\begin{cor}
    \label{cor-coeff-prim-relations-give-non-zero-admissible-paths}
Any $Z$-admissible path lies outside the ideal $I=I_{Z}+I_{V}+I_{\neg V}$ of $RQ$. 
\end{cor}

It is convenient here to state some further consequences of \Cref{lem-canonical-form-in-I}. 
These consequences, namely \Cref{lem-string-multiserial-module-property}  and \Cref{lem-string-multiserial-module-property-dual}, are used in the proof of \Cref{thm-examples-of-string-algebras}.

\begin{lem}
\label{lem-string-multiserial-module-property} 
Let $I=I_{Z}+I_{V}+I_{\neg V}$, let $u$ and $w$ be vertices, let $a_{1},\dots,a_{m}$ be the distinct arrows with $t(a_{l})=w$ for each $l=1,\dots,m$, and let $x_{1},\dots,x_{m}\in e_{u}RQ$ where $x_{l}a_{l}=\sum_{q\in Q_{Z}(u\leftarrow w)}t_{q,l}q$  for some $t_{q,l}\in R$. 
If $\sum_{l=1}^{m}x_{l}a_{l}\in I$ then $x_{l}a_{l}\in I$ for each $l$. 
\end{lem}

\begin{proof}
Firstly note that $\sum_{l=1}^{m}x_{l}a_{l}\in e_{u}Ie_{w}$ by definition, and so  $\sum_{l=1}^{m}x_{l}a_{l}=\sum_{p}t_{p}p$ where
\[
\begin{array}{c}
t_{p}=
\begin{cases}
-
r_{i,c^{-1}p}+\sum_{j=1}^{n}
s_{j}r_{j,p}
  & \left(\exists\, c\in Q_{Z}^{\,\prime}(u\circlearrowleft) :p\in Q^{+}_{Z}(c\mid\leftarrow  w),u\in V[i]
\right)
\\
\sum_{j=1}^{n}
s_{j}r_{j,p} & \left(\text{otherwise}\right).
\end{cases}
\end{array}
\]
for some $r_{j,p}\in R$ with $j=1,\dots,n$ and $p \in Q_{Z}(u\leftarrow w)$ by \Cref{lem-canonical-form-in-I}. 
Since the paths in $Q$ form a basis for $RQ$ we have $t_{p}=\sum_{l}t_{p,l}$ for each $p$.  
Furthermore $t_{p,l}=0$ for any $l$ and $p$ that does not have right arrow $a_{l}$. 
Let $r_{j,q'}=0$ for all paths $q'\notin Q_{Z}(u\leftarrow w)$.

Let $e$ be the trivial path at $u$.  
We assert that $r_{i,e}e\in I$ for each $i$. 
If $w\neq u$ then $r_{j,e}=0$. 
Hence for our assertion it suffices to assume $u=w$. 
Since $0=t_{e}=\sum_{j=1}^{n}s_{j}r_{j,e}=0$ we have $s_{i}r_{i,e}\in \mathfrak{a}$ where  $\mathfrak{a}=\langle s_{j}\colon i\neq j=1,\dots,n\rangle $, so  $r_{i,e}\in \mathfrak{a}$ since  $(s_{1},\dots,s_{n})$ is regular. 
If $u\notin V$ then $s_{j}e\in I_{\neg V}$ for all $j$, giving $r_{i,e}e\in I_{\neg V}$ since $r_{i,e}\in \mathfrak{a}$. 
Otherwise $u\in V[i]$ for some $i$. 
Again this means $r_{i,e}e\in I_{\neg V}$ since $r_{i,e}\in \mathfrak{a}$ and $s_{j}e\in I_{\neg V}$ for each $j\neq i$. 

Hence our assertion holds, which means we can assume that $r_{i,e}=0$. Define $r_{j,p,l}\in R$ for each $j$, $p$ and $l$ as follows. 
If $p$ is non-trivial and has right arrow $a_{l}$ we let $r_{j,p,l}=r_{j,p}$. In all other cases we let $r_{j,p,l}=0$. 
Now define $t'_{p,l}\in R$ for each $p$ and each $l$, by
\[
\begin{array}{c}
t'_{p,l}=
\begin{cases}
-
r_{i,c^{-1}p,l}+\sum_{j=1}^{n}
s_{j}r_{j,p,l}
  & \left(\exists\, c\in Q_{Z}^{\,\prime}(u\circlearrowleft) :p\in Q^{+}_{Z}(c\mid\leftarrow  w),u\in V[i]
\right)
\\
\sum_{j=1}^{n}
s_{j}r_{j,p,l} & \left(\text{otherwise}\right).
\end{cases}
\end{array}
\]
By construction $t_{p}=\sum_{l=1}^{m}t_{p,l}$ for each $p$. 
Furthermore, by identifying expressions in $RQ$ with right arrow $a_{l}$, we have $x_{l}a_{l}=\sum_{p}t_{p,l}p$ for each $l$, and so $x_{l}a_{l}\in I$ by \Cref{lem-canonical-form-in-I}. 
\end{proof}

\begin{lem}
\label{lem-canonical-form-in-I-dual}
Let  $u$ and  $w$ be vertices  and 
 $z\in  e_{u}(I_{Z}+I_{V}+I_{\neg V})e_{w}$. There exist $r_{j,q}\in R$ for $j=1,\dots,n$ and $q\in Q_{Z}(u\leftarrow w)$ such that $z-\sum_{q}
t_{q}q\in I_{Z} $  where
\[
\begin{array}{c}
t_{q}=
\begin{cases}
-
r_{i,qd^{-1}} +\sum_{j=1}^{n}
s_{j}r_{j,q}
 & \left(\exists\, d\in Q_{Z}^{\,\prime}(w\circlearrowleft)\, :\,q\in Q^{+}_{Z}(u\leftarrow \mid\,d)
\right)
\\
\sum_{j=1}^{n}
s_{j}r_{j,q} & \left(\text{otherwise}\right).
\end{cases}
\end{array}
\]
Conversley we have $\sum_{q}t_{q}q\in e_{u}(I_{Z}+I_{V}+I_{\neg V})e_{w}$ for any $r_{j,q}\in R$ defining $t_{q}\in R$ as above.
\end{lem}

\Cref{lem-canonical-form-in-I-dual} is dual to \Cref{lem-canonical-form-in-I}, hence we omit the proof. 
By instead applying \Cref{lem-canonical-form-in-I-dual} in the proof of \Cref{lem-string-multiserial-module-property}, we obtain \Cref{lem-string-multiserial-module-property-dual}.

\begin{lem}
\label{lem-string-multiserial-module-property-dual} 
Let $I=I_{Z}+I_{V}+I_{\neg V}$, let $u$ and $w$ be vertices, let $a_{1},\dots,a_{m}$ be the distinct arrows with $h(a_{l})=u$ for each $l=1,\dots,m$, and let $x_{1},\dots,x_{m}\in RQe_{w} $ where $x_{l}a_{l}=\sum_{q\in Q_{Z}(u\leftarrow w)}t_{q,l}q$  for some $t_{q,l}\in R$. 
If $\sum_{l=1}^{m}a_{l}x_{l}\in I$ then $x_{l}a_{l}\in I$ for each $l$. 

\end{lem}

\section{String algebras over local rings}
\label{sec-string-algs}

For the definition of a string algebra over a local ring we start with a general set up. 

\begin{setup}
    For the beginning of  \S\ref{sec-string-algs} we ignore the various assumptions  built up in  \S\ref{sec-primitive-cyces-and-nerve-partitions}, and we just let $R$ be a noetherian local ring and $Q$ be a quiver.     
    This is so that we can state necessary definitions and remarks in the appropriate generality. 
\end{setup}

\begin{notn}
\label{notn-start}
    For an ideal $\mathfrak{a}$ in $R$ let $\mathfrak{a} Q$ be the $R$-submodule of $RQ$ consisting of $R$-linear combinations of paths with coefficients in $\mathfrak{a}$. 
Let $A_{\ell}\subseteq RQ$ be the $R$-span of the paths of length $\ell \geq 0$ and let $A=\bigoplus_{\ell\geq 0}A_{\ell}$. 
\end{notn}


We now recall the definition of a string algebra over a local ring from \cite{benn-tenn-string-alg-over-loc}. 

\begin{defn}
\label{def-string-algebra}
\cite[Definition 1.4]{benn-tenn-string-alg-over-loc} A \emph{string algebra over} $R$ is an $R$-algebra of the form $\Lambda =RQ/I$ for an ideal $I$ of $RQ$ where the following conditions hold. 
\begin{itemize}
    \item $Q$ is \emph{biserial}, meaning any vertex is the head of at most $2$ arrows and the tail of at most $2$ arrows. 
    \item \cite[Definition 5.1]{benn-tenn-string-alg-over-loc} $I$ is \emph{special}, meaning that conditions (1) and (2) below hold.
    \begin{enumerate}
        \item For any arrow $b$ there exists at most $1$ arrow $a$ such that $t(a)=h(b)$ and  $ab\notin I$. 
        \item  For any arrow $b$ there exists at most $1$ arrow $c$ such that  $h(c)=t(b)$ and  $bc\notin I$.
    \end{enumerate}
    \item (c.f. \cite{CarSal1987}) \cite[Definition 4.2]{benn-tenn-string-alg-over-loc} $I$ is \emph{admissible}, meaning that conditions (3) and (4) below hold.  
        \begin{enumerate}
        \setcounter{enumi}{2}
        \item The ideal  $I$ is  \emph{bounded above}, meaning $I\subseteq A+ \maxideal Q$ and $I\cap A \subseteq A^{2}+ \maxideal Q\cap A$. 
        \item The ideal  $I$ is \emph{bounded below}, meaning  $A^{m}\subseteq I+ \maxideal Q$ and $\maxideal Q\subseteq I+\sum_{\ell=1}^{h} A_{\ell}$ for some $m,h>0$. 
\end{enumerate} 
    \item  \cite[Definition 5.9]{benn-tenn-string-alg-over-loc} $I$  is    \emph{arrow}-\emph{direct}, meaning that conditions (5) and (6)   below hold. 
        \begin{enumerate}
        \setcounter{enumi}{4}
        \item For any arrow $a$ we have $\Lambda a\cap \sum \Lambda b=0$ where the sum runs over arrows $b\neq a$ with $t(b)=t(a)$. 
        \item For any arrow $a$ we have $a\Lambda \cap\sum b\Lambda =0$ where the sum runs over arrows $b\neq a$ with $h(b)=h(a)$.
        \end{enumerate}
\end{itemize}
\end{defn}

We recall some relations between conditions (1)--(6) in \Cref{def-string-algebra} from \cite{benn-tenn-string-alg-over-loc}. 

\begin{rem} 
\label{rem-sufficient-conditons-for-main-theorem}
Later we check the technical conditions  from \Cref{def-string-algebra} hold in the setting built up in \S\ref{sec-primitive-cyces-and-nerve-partitions}. Here we discuss them.  
Note firstly that the inclusion $\Lambda a\cap \sum \Lambda b=0$ is equivalent to the inclusion 
$ RQ a \cap \sum RQ b\subseteq I
$. 
Likewise the inclusion $a\Lambda \cap\sum b\Lambda =0$ is equivalent to the inclusion $aRQ  \cap \sum bRQ \subseteq I$.    

Following \cite[Definition 4.13]{benn-tenn-string-alg-over-loc}, an ideal $I$ is  \emph{arrow}-\emph{distinct} if (5') and (6')   below hold.  
        \begin{enumerate}
        \item[(5')]  For any arrow $a$ we have $\Lambda a\cap \sum \Lambda b\subseteq \rad(\Lambda a)$ where the sum runs over arrows $b\neq a$ with $t(b)=t(a)$. 
        \item[(6')] For any arrow $a$ we have $a\Lambda \cap \sum b\Lambda \subseteq \rad(a\Lambda )$ where the sum runs over arrows $b\neq a$ with $h(b)=h(a)$.
        \end{enumerate}
    Clearly condition (5') (respectively, (6')) implies conditon (5) (respectively, (6)) from \Cref{def-string-algebra}.  
    Following \cite[Definition 4.13]{benn-tenn-string-alg-over-loc}, an ideal $I$ is  \emph{permissible} if $a\notin I$ for any arrow $a$. Provided  $I$ is bounded below, then $I$ is also bounded above if and only if $I$ is arrow-distinct and permissible; see \cite[Theorem 1.2 (3)]{benn-tenn-string-alg-over-loc}. 

    Let $Z$ be a set of paths and $I$ be an ideal with $\langle Z\rangle\subseteq I$. 
    It is immediate that if (SP1) (respectively, (SP2)) from 
    \Cref{defn-spec-quiver-and-relations} holds, then (1) (respectively, (2)) from \Cref{def-string-algebra} holds. 
\end{rem}

We note some consequences of assuming $R$ is $\maxideal$-adically complete. 

\begin{rem}
    \label{rem-completeness-gives-completeness} Let $\Lambda$ be a string algebra over a local ring $(R,\maxideal,k)$. 
    By definition this means $\Lambda=RQ/I$ where $I$ is an admissible ideal of $RQ$. 
    By \cite[Theorem 1.2(1)]{benn-tenn-string-alg-over-loc} this means $\Lambda$ is module-finite over $R$, and that 
 the jacobson radical $\jacrad$ of $\Lambda$ is generated by the arrows in $Q$. 
 By \cite[Proposition 20.6]{Lam1991} this means $\jacrad^{m}\subseteq \mathcal{M} \subseteq \jacrad$ for some integer $m>0$ where $\mathcal{M}=\Lambda \maxideal$. 
 
 Said another way, the $\jacrad$-adic topology on the ring $\Lambda$ is equivalent to the $\mathcal{M}$-adic topology. 
    This means that if $R$ is $\maxideal$-adically complete then  $\Lambda$ is $\mathcal{M}$-adically complete, and hence $\Lambda$ is $\jacrad$-adically complete; see for example \cite[Proposition 21.34]{Lam1991}.  
    Later we revisit this remark for the complete local ring $K[\vert t_{1},\dots,t_{n}\vert ]$. 
\end{rem}

In \S\ref{admissibility} we apply  \Cref{rem-sufficient-conditons-for-main-theorem} together with other results to complete the proof of   \Cref{thm-examples-of-string-algebras}. 
In \S\ref{sec-proof-of-ricke-examples} we likeiwse apply \Cref{rem-completeness-gives-completeness} in the proof of \Cref{cor-rickes-examples}.

\subsection{Proof of Theorem 1.1. }
\label{admissibility} To prove \Cref{thm-examples-of-string-algebras} we set up the hypothesis. 

\begin{setup}
\label{setup-main-theorem}
    In  \S\ref{admissibility} we let $Z$ be a set of paths of length at least $2$, we assume that $(Q,Z)$ is special, we let $V=V[1]\sqcup \dots\sqcup V[n]$ be the primitive-nerve partition associated to $(Q,Z)$,  we let $( s_{1},\dots,s_{n}) $ be a regular sequence in $R$ such that  $\maxideal=\langle s_{1},\dots,s_{n}\rangle$, and we recall that for each $v\in V$ we write  $\sigma_{v}\in RQ$ for  sum of the $Z$-primitive cycles incident at $v\in V$. 
    Finally, we define the ideal $I$ of $RQ$ by 
    \[
\begin{array}{c}
 I = \langle Z\rangle 
+
\sum_{i=1}^{n}
\langle s_{i}e_{v}-\sigma_{v}\,\mid\, v\in V[i]\rangle 
+
\langle s_{i}e_{v}\,\mid\, v\in Q_{0}\setminus V[i]\rangle. 
\end{array}
 \]
  Hence   $I=I_{Z}+I_{V}+I_{\neg V}$ from \Cref{notation-ideal-I}.
\end{setup}

We note that \Cref{setup-main-theorem} includes and adds to \Cref{setup-R-m-k}, \Cref{setup-scalars}, and \Cref{setup-regular-scalars} from \S\ref{sec-primitive-cyces-and-nerve-partitions} and \S\ref{subsec-prim-reg-rad-collections}.

\begin{lem}
\label{lem-coeff-prim-relations-give-module-finite-algebras}
If $i=1,\dots,n$ and $v\in V[i]$ then (1) and (2) hold for all $m\geq h$ for some $h>0$. 
\begin{enumerate}
        \item If  $p\in Q_{Z}^{m}( \leftarrow v)$ then  $p-s_{i}^{t}q\in I$ for some $t\geq 1$ and $q\in Q_{Z}^{\ell}(\leftarrow v)$ with $0<\ell\leq h$.
        \item If  $p\in Q_{Z}^{m}( v\leftarrow )$ then  $p-s_{i}^{t}q\in I$ for some $t\geq 1$ and $q\in Q_{Z}^{\ell}(v\leftarrow )$ with $0<\ell\leq h$.
\end{enumerate}
In particular  $A^{h}e_{v}+e_{v}A^{h}\subseteq I+ \mathfrak{a}Q$ and $\mathfrak{a}Qe_{v}+e_{v}\mathfrak{a}Q\subseteq I+\sum_{\ell=1}^{h}A_{\ell}$ where $\mathfrak{a}=\langle s_{i}\rangle$. 
\end{lem}

\begin{proof}
(1) Let $i=1,\dots,n$ and let $s=s_{i}$.
We can assume $Q^{d}_{Z}(\leftarrow v)\neq \emptyset$ for all $d\geq 0$. 
Let $p\in Q_{Z}^{m}( \leftarrow v)$. 
Let $\sigma=\sigma_{v}$. 
By \Cref{lem-special-paths-of-arbitrary-length-initial-subpaths-or-primitive-cycles} there exists an integer $h'>0$ such that 
for each $r\geq h' $ every path $p'\in  Q_{Z}^{r}( \leftarrow v)$  must be a right subpath of a power of some $c\in Q_{Z}^{\,\prime}(v\circlearrowleft)$. 

By \Cref{cor-bound-on-number-of-primitive-cycles} the set $Q_{Z}^{\,\prime}(v\circlearrowleft)$ is finite. 
So we can define an integer $h>h'$ such that the length of each cycle $c\in Q_{Z}^{\,\prime}(v\circlearrowleft)$  is less than $\frac{1}{2}h$. 
Since the  $Z$-admissible path $p$ has length  $m>h>h'$, as above we must have $p=q'c^{d}$ for some $c\in Q_{Z}^{\,\prime}(v\circlearrowleft)$, some $d\geq 1$ and some right subpath $q'$ of $c$. 
Note that  $d\geq 2$ for otherwise $d=1$ and then $p=q'c$ would have length less than $2\times \frac{1}{2}h=h$. 
Let $t=d-1$ and $q=q'c$. 

Let $c'\in Q_{Z}^{\,\prime}(v\circlearrowleft)$ with $c'\neq c$. 
By \Cref{lem-primitive-cycles-sharing-arrows} $c'$ and $c$ have distinct left arrows, meaning $Z\mid c'c$ by (SP1) and since $Z\nmid cc$. 
Running through all such $c'\neq c$ we have that $c^{t+1}-\sigma^{t} c$ is a sum of paths that factor through $Z$, and so $c^{t+1}-\sigma^{t} c\in I$. 
Also note that $s\sigma=\sigma s$    and so for any $j>0$ we have $s^{j}e_{v}-\sigma^{j}=(se_{v}-\sigma)\sum_{i=0}^{j-1}s^{j-1-i}e_{v}\sigma^{i}\in I$. 
Altogether,
\[
p-s^{t}q=q'c^{t+1}-s^{t}q=q'(c^{t+1}-s^{t}c)=q'((c^{t+1}-\sigma^{t}c)-c(s^{t}e_{v}-\sigma^{t}))\in I.
\]
This completes the proof of (1). 
The proof of (2) is similar. 
We now prove the final statement. 
The $R$-module $A^{h}e_{v}$ is generated by the union $\bigcup _{m\geq h}Q_{Z}^{m}( \leftarrow v)$, and by (1) we have that any element in this union has the form $x+s^{t}q$ where $x\in I$ and $q\in \bigcup _{\ell\leq h}Q_{Z}^{\ell}( \leftarrow v)$. 

Since $s^{t}q\in \mathfrak{a}Q$ this shows $A^{h}e_{v}\subseteq I+\mathfrak{a}Q$. 
Similarly, by (2) we have $e_{v}A^{h}\subseteq I+\mathfrak{a}Q$. 
Again by (1), if $r\in R$, $j>0$ and $p\in \bigcup _{m\geq h}Q_{Z}^{m}( \leftarrow v)$ then $rs^{j}p=x+rs^{j+t}q$ for $x$, $t$ and $q$ as above. 
This shows $\mathfrak{a}Qe_{v}\subseteq I+\sum_{\ell=1}^{h}A_{\ell}$, and similarly $e_{v}\mathfrak{a}Q\subseteq I+\sum_{\ell=1}^{h}A_{\ell}$ by (2). 
\end{proof}

We now combine many of the results built up so far to check conditions from \Cref{def-string-algebra} hold. 

\begin{cor}
    \label{cor-bounded-below}
    The ideal $I$ from \Cref{setup-main-theorem}  is bounded below in sense of \Cref{def-string-algebra}(4). 
\end{cor}

\begin{proof}
    For each vertex $v$ and $i=1,\dots,n$, by \Cref{lem-coeff-prim-relations-give-module-finite-algebras} we have an integer $h(v,i)>0$ such that
    \[
    \begin{array}{cc}
    A^{h(v,i)}e_{v}+e_{v}A^{h(v,i)}\subseteq I+ \langle s_{i}\rangle Q, 
    &
    \langle s_{i}\rangle Qe_{v}+e_{v}\langle s_{i}\rangle Q\subseteq I+\sum_{\ell=1}^{h(v,i)}A_{\ell},
    \end{array}
    \]
    Since $\maxideal=\langle s_{1},\dots,s_{n}\rangle$, it suffices to let $h=m=\max \{h(v,i)\colon i=1,\dots,n,v\in Q_{0}\}$ in \Cref{def-string-algebra}(4). 
\end{proof}

\begin{lem}
        \label{cor-arrow-direct}
    The ideal $I$ from \Cref{setup-main-theorem}  is arrow-direct in sense of \Cref{def-string-algebra}(5, 6). 
\end{lem}

\begin{proof}
    Let $a$ be an arrow with tail $w$. 
    Suppose $ x_{a}a +I = \sum_{b}x_{b}b +I$ for some $x_{a},x_{b}\in RQ$ where $b$ runs through the arrows with tail $w$ that are distinct from $a$. 
    Let $z=x_{a}a-\sum_{b}x_{b}b$, and so $z\in I$. 
    If $u$ is a vertex then  $e_{u}z\in e_{u}Ie_{w}$. 
    By \Cref{lem-string-multiserial-module-property} this means $e_{u}x_{a}a\in e_{u}I e_{w}$. 
    Thus $x_{a}a\in I$, and this shows that the intersection $\Lambda a\cap \sum_{b} \Lambda b$ is trivial, where $\Lambda=RQ/I$. 
    Thus condition (5) from \Cref{def-string-algebra} holds. 
    The proof that condition (6) from \Cref{def-string-algebra} holds is similar, but uses \Cref{lem-string-multiserial-module-property-dual}. 
\end{proof}

\begin{cor}
    \label{lem-bounded-above}
    The ideal $I$ from \Cref{setup-main-theorem}  is bounded above in sense of \Cref{def-string-algebra}(3). 
\end{cor}

\begin{proof}
    By \Cref{cor-arrow-direct} we have that $I$ is arrow-direct. 
    In particular $I$ is arrow-distinct, meaning that conditions (5') and (6') from  \Cref{rem-sufficient-conditons-for-main-theorem} hold, since conditions (5) and (6) from \Cref{def-string-algebra} hold.  
    As discussed in \Cref{rem-sufficient-conditons-for-main-theorem}, to see that condition (3) from \Cref{def-string-algebra} holds, it suffices to show that $a\notin I$ for any arrow $a$. 
    Since every element of $Z$ is a path of length at least $2$ by \Cref{setup-main-theorem}, this is immediate by \Cref{cor-coeff-prim-relations-give-non-zero-admissible-paths}. 
\end{proof}


\begin{proof}[Proof of \Cref{thm-examples-of-string-algebras}.]
    The first statement is precisely the statement of \Cref{cor-coeff-prim-relations-give-non-zero-admissible-paths}. 
    We already assumed that $(Q,Z)$ is special, meaning that conditions (1) and (2) from \Cref{def-string-algebra} hold, by \Cref{rem-sufficient-conditons-for-main-theorem}. 
    
    Assuming additionally that $(Q,Z)$ is special biserial means that $Q$ is biserial. Finally: by \Cref{lem-bounded-above} condition (3) from \Cref{def-string-algebra} holds; by \Cref{cor-bounded-below} condition (4) from \Cref{def-string-algebra} holds; and by \Cref{cor-arrow-direct} conditions (5) and (6) from \Cref{def-string-algebra} hold. 
\end{proof}

\subsection{Proof of Proposition 1.4. } 
\label{sec-proof-of-ricke-examples}

\begin{setup}
    \label{sec-proof-of-ricke-examples-again}
    In \S\ref{sec-proof-of-ricke-examples} we keep and add to the assumptions and notation of \Cref{setup-main-theorem}. 
    We specify here by setting  $R=k[\vert t_{1},\dots , t_{n}\vert ]$, the power series algebra in $n$-variables over the field $k$. 
    Let $H$ be the ideal in the path algebra $KQ$ generated by $Z$, and let $B$ be the ideal in $KQ$ generated by the arrows. 
\end{setup}

\begin{defn}
    \label{def-ricke-completed-string-algebras}
    \cite[\S3.2]{Ric2017} 
    A \emph{completed string algebra} is a quotient $\overline{KQ}/\overline{H}$ where $Q$ is biserial, $(Q,Z)$ is  special, $\overline{kQ}$ is the completion of $kQ$ in the $B$-adic topology and where $\overline{H}=\bigcap_{d>0}(H+B^{d})$. 
    Note that $\overline{kQ}/\overline{H}$ is the completion of the quotient $kQ/H$ in the $\mathcal{B}$-adic topology where $\mathcal{B}=B+H/H$. 
\end{defn}

In \Cref{lem-existence-of-the-maps} we observe the existence of a ring isomorphism. 
This result is a key observation for the proof of \Cref{cor-rickes-examples}, where we consider inverse limits of these rings. 
For the proof of \Cref{lem-existence-of-the-maps} we use \Cref{thm-coeff-prim-relations-give-non-zero-admissible-paths}.  

\begin{lem}
    \label{lem-existence-of-the-maps}
    For each integer $d>0$ there is an $R$-algebra isomorphism $ kQ/(H+B^{d})\to RQ/(I+A^{d})$ induced by the $k$-algebra embedding $kQ\to RQ$. 
\end{lem}

\begin{proof}
    By \Cref{sec-proof-of-ricke-examples-again} we have that $k$ is a subring of $R$, and so $kQ$ is a $k$-subalgebra of $RQ$. 
    Since $B$ is the ideal in $kQ$ generated by the arrows, and since $A$ is the ideal in $RQ$ generated by the arrows, we have the inclusion $A\subseteq B$ of $k$-vector spaces. 
    Likewise, recalling that $I=I_{Z}+I_{V}+I_{\neg V}$ where $I_{Z}$ is the ideal in $RQ$ generated by $Z$, and considering that $H$ is the ideal in $kQ$ generated by $Z$, we have $H\subseteq I_{Z}$ as $k$-vector spaces. 
    Hence for each $d>0$ we have $H+B^{d}\subseteq I+A^{d}$ meaning that the $k$-algebra embedding induces a $k$-algebra homomorphism $\varphi_{d}\colon kQ/(H+B^{d})\to RQ/(I+A^{d})$. 
    It remains to show $\varphi_{m}$ is bijective. 
    
    Fix $d$ and let $\varphi=\varphi_{d}$. 
    We start by proving $\varphi$ is injective, for which it suffices to fix vertices $u$ and $w$ and then prove $e_{u}(kQ\cap (A^{d}+I))e_{w}\subseteq e_{u}(B^{d}+H)e_{w}$. 
    So suppose $x+y\in e_{u}kQe_{w}$ where $x\in e_{u}B^{d}e_{w}$ and $y\in e_{u}Ie_{w}$. 
    Write $y=\sum r_{q}q$ where $r_{q}\in R$ and where $q$ runs through the paths with $t(q)=w$ and $h(q)=u$. 
    
    Define $r'_{q}\in R$ for each such $q$  by setting $r'_{q}=r_{q}$ when $q\in Q_{Z}(u\leftarrow w)$, meaning when $q$ is $Z$-admissible, and let $r'_{q}=0$ otherwise. 
    Let $y'=\sum r'_{q}q$, and note that in the notation so far we have 
    \[
    \begin{array}{cc}
        \sum_{q\in Q_{Z}(u\leftarrow w)}r_{q}q=y'=  (y'-y)+y=(y'-y)+(x+y)-x\in e_{u}Ie_{w}.
        & 
        (*)
    \end{array}
    \] 
    For the moment suppose $y'\neq 0$. 
    Choose a path $p\in Q_{Z}(u\leftarrow w)$ of minimal length $\ell$ such that $r_{p}\neq 0$. 
    By \Cref{thm-coeff-prim-relations-give-non-zero-admissible-paths} we have that $r_{p}\in \maxideal$. 
    We claim that  $\ell\geq d$. 
    For a contradiction suppose $\ell<d$. 
    Considered as an element of $RQ$, the coefficient of $p$ in $x\in A^{d}$ must therefore be $0$. 
    Since $p$ is $Z$-admissible, and since $y'-y$ is an $R$-linear combination of $Z$-inadmissible paths, the coefficient of $p$ in $y'-y$ must also be $0$. 
    But by $(*)$ this means $r_{p}$ is equal to the coefficient of $p$ in $x+y\in kQ$, and so $0\neq r_{p}\in k\cap \maxideal=0$, a contradiction. 
    
    Hence our claim holds. 
    Consider the ring homomorphism $R\to k$ given by sending any power series $f(t_{1},\dots,t_{n})\in R$ to its constant term, that is, its evaluation $f(0,\dots,0)$ at $t_{1}=\dots=t_{n}=0$. Consider the induced ring homomorphism $\zeta\colon RQ\to kQ$, and note that the restriction of $\zeta$ to $kQ$ is the identity.  
    Note that the image of $A^{d}$ under $\zeta$ is equal to $B^{d}$. 
    Likewise the image of $I_{Z}$ under $\zeta$ is  equal to $H$. 
     By our claim we have $y'\in A^{d}$ which means $x+y'\in A^{d}$ and, by construction, $y-y'\in I_{Z}$. 
    Altogether we have 
    \[
    kQ\ni x+y=\zeta(x+y)=\zeta(x+y' + y - y')=\zeta(x+y')+\zeta(y-y')\in B^{d}+H.
    \]
    This completes our proof that $\varphi$ is injective. We secondly prove $\varphi$ is surjective. 
    It suffices to let $r\in R$ and $p$ be any $Z$-admissible path of length less than $d$, and then prove that  $rp+B^{d}+I$ is in the image of $\varphi$. 
    Let $v=h(p)$. 
    In the situation we are in, the element $r\in R$ is a power series $r=f(t_{1},\dots,t_{n})$. 
    Note that if $v\notin V$ then $rp+I+A^{m}=\zeta(r)p+I+A^{m}$ which is the image of $\zeta(r)p+H+B^{m}$ and there is nothing to prove.  
    Assume otherwise, so that we have $v\in V[i]$ for some $i=1,\dots,n$. 
    In this case $t_{j}p\in I_{\neg V}$ for all $j\neq i$. 
    Consider the power series $g(t_{i})\in k[\vert t_{i} \vert]$ defined by taking $t_{j}=0$ in $f(t_{1},\dots,t_{n})$ for each $j\neq i$. 
    So there is a scalar $\lambda_{h}\in k$ for each integer $h\geq 0$ such that 
    $g(t_{i})=\sum_{h}\lambda_{h}t_{i}^{h}$. 
    Note that $tp+I=\sigma_{v}p$. 
    Furthermore there is an integer $\ell$ such that the length of $\sigma^{h}_{v}p\in A^{d}$ for each $h>\ell$. 
    Altogether this means 
    \[
    \begin{array}{c}
    rp+I+A^{d}=g(t_{i})p+I+A^{d}=\sum_{h=0}^{\ell}\lambda_{h}\sigma_{v}^{h} p+I+A^{d}=
    \varphi
    \left(
    \sum_{h=0}^{\ell}\lambda_{h}\sigma_{v}^{h} p+H+B^{d}
    \right)
    \end{array}
    \]
    as required for the surjectivity of $\varphi$. 
\end{proof}

\begin{rem}
\label{rem-complete-ring-isos}
    Let $\Psi$ be a ring and $\Omega$ be rings, let $\mathcal{M}$ be an ideal in $\Psi$ and let $\mathcal{N}$ be an ideal in $\Omega$. 
    Suppose that for each integer $d>0$ there is a ring isomorphism $\chi_{d}\colon \Psi/\mathcal{M}^{d}\to\Omega/\mathcal{N}^{d}$ satisfying the following condition.
    \[
    \begin{array}{cc}
    (*)_{d}
    &
    \text{If $\chi_{d}(\psi+\mathcal{M}^{d})=\omega+\mathcal{N}^{d}$,  $\chi_{d+1}(\psi'+\mathcal{M}^{d+1})=\omega'+\mathcal{N}^{d+1}$ and $\psi-\psi'\in\mathcal{M}^{d+1}$ then $\omega-\omega'\in\mathcal{N}^{d}$.}  
    \end{array}
    \]
    Then the $\mathcal{M}$-adic completion of $\Psi$ is isomorphic to the $\mathcal{N}$-adic completion of $\Omega$. 
    Too see this is a straightforward application of the condition $(*)_{d}$ together with the universal property of the involved (inverse) limits. 
\end{rem}

\begin{proof}[Proof of \Cref{cor-rickes-examples}.]
    Recall $\jacrad=\rad(RQ/I)$ is the ideal generated by the arrows in $Q$. 
    By \Cref{def-ricke-completed-string-algebras} the completed string algebra $\overline{KQ}/\overline{H}$ is isomorphic to the $\mathcal{B}$-adic completion $\overline{\Gamma}$ of $\Gamma=kQ/H$ where $\mathcal{B}=B+H/H$, the ideal in $\Gamma$ generated by the arrows. 
    Note that for each integer $d>0$ we have ring isomorphisms
    \[
    \begin{array}{ccc}
    \Lambda /\jacrad^{d}=\dfrac{RQ/I}{(A+I)^{d}/I}\cong \dfrac{RQ}{A^{d}+I},
         &  
     \Gamma/\mathcal{B}^{d}=\dfrac{kQ/H}{(B+H)^{d}/H}\cong \dfrac{kQ}{B^{d}+H}. 
    \end{array}
    \]
    By \Cref{lem-existence-of-the-maps} there is an isomorphism between these quotients satisfying condition $(*)_{d}$ of \Cref{rem-complete-ring-isos}. 
    Thus we have that the $\jacrad$-adic completion of $\Lambda$ is isomorphic to $\overline{\Gamma}$, and since $R$ is $\maxideal$-adically complete this means $\Lambda$ is isomorphic to $\overline{KQ}/\overline{H}$ by \Cref{rem-completeness-gives-completeness}. 
\end{proof}

\section{Examples of string algebras over regular local rings}
\label{sec-examples}





We use \Cref{thm-examples-of-string-algebras} to generate examples of string algebras over regular local rings.  
Regular local rings with interesting properties already appear in literature, see for example: the book by Nagata \cite[Appendix A1]{Nagata-local-rings}; the article of Valabrega \cite{Valabrega-regular-local-rings-and-excellent-rings}; and the more recent article of Nishimura \cite[Examples 4.3--4.5]{Nishimure-A-few-examples-of-local-rings-I}. 
For brevity and simplicity we consider complete regular local rings, whose structure is determined by the Cohen structure theorem. 
Note that examples in case $R=k[\vert t_{1},\dots,t_{n}\vert ]$ were given by \Cref{cor-rickes-examples}.

In each example we need only specify  the quiver $Q$ and relations $Z$ defining the special pair $(Q,Z)$, the primitive-nerve partition $V=V[1]\sqcup\dots \sqcup V[n]$, the local ring $R$ of Krull dimension $n$, and the regular sequence of generators $(s_{1},\dots,s_{n})$ of the maximal ideal $\maxideal$. 

We start by considering a discrete valuation domain. By doing so we begin by recovering examples over  from \cite[\S6]{benn-tenn-string-alg-over-loc}. 
We then generate other  more complicated examples

\subsection{Discrete valuation rings}
\label{subsec-examples-dvr}

In \Cref{subsec-examples-dvr} we let $R$ be a discrete valuation ring, meaning that as a local noetherian ring $R$ is integrally closed and has Krull dimension $1$. 
In order to apply \Cref{thm-examples-of-string-algebras}, for the choices of $Q$ and $Z$ we make, we must have that primitive-nerve partition has $1$ part. 
In other words, we need to choose examples of $Q$ and $Z$ such that (the pair $(Q,Z)$ is special, and such that) for any $Z$-primitive cycles $c_{0}$ and $c_{m}$ there exist $Z$-primitive cycles $c_{1},\dots,c_{m-1}$ such that $V(c_{i})\cap V(c_{i+1})\neq \emptyset$ for each $i=0,\dots, m-1$.  

We will write $\pi$ for the \emph{uniformizer}, that is, the single generator of the maximal ideal $\maxideal$. 
Because of what has been discussed above, note that for each vertex $v\in V$ the element $\pi e_{v}+I\in \Lambda$ acts as the sum of all of the $Z$-primitive cycles incident at $v$. 

\subsubsection{Examples from \cite[\S6.3]{benn-tenn-string-alg-over-loc}}

Ignoring the top row that labels the columns, in rows $1$--$5$ 
 of the following table we recover the example of a string algebra over $R$ from  \cite[\S6.3.1--\S6.3.5]{benn-tenn-string-alg-over-loc}, respectively, by \Cref{thm-examples-of-string-algebras}. 
 Where we note a presentation of the corresponding string algebra, in the right-most column, the overlaying equalities should be read as \emph{identification in the residue field}. 
 Specifically, for example, in the second row one is considering the subring of $R\times R$ given by pairs $(r,r')$ with $r,r'\in R$ such that $r-r'\in \maxideal$. This notation follows that adopted by Ringel and Roggenkamp \cite[(2.14)~Remarks~(iv)]{RinRog1979} and  Wiedemann \cite[p.~314]{Wie1986}.  
\begin{table}[H]
    \centering
    \begin{tabular}{|c|c|c|c|}
    \hline
       Quiver $Q$  
       & 
       Zero-relations $Z$
       & 
       $Z$-primitive cycles 
       &
       String algebra $\Lambda=RQ/I$
       \\
       \hline 
       \begin{tikzcd}[column sep=0.7cm, row sep=0.3cm]
1\arrow[out=150,in=210,loop,  "", "a"',distance=0.6cm]
\arrow[r, bend right, swap,"b"]
&
2\arrow[l,"c", swap, bend right]
\end{tikzcd}
       &
       $bc, \,\, a^{2}$
       &
       $acb,\,\,cba,\,\,bac$
       &
$\left( \begin{smallmatrix} 
\tikznode{15}{ {$R$}} \hspace{2mm} & 
\tikznode{16}{ {$\maxideal$}} \hspace{2mm} & 
\tikznode{17}{ {$\maxideal$}} \\
\tikznode{27}{ {$R$}}\hspace{2mm}  &
\tikznode{28}{ {$R$}}\hspace{2mm}  & 
\tikznode{29}{ {$\maxideal$}} \\
\tikznode{37}{ {$R$}}\hspace{2mm}  &
\tikznode{38}{ {$R$}}\hspace{2mm}  & 
\tikznode{39}{ {$R$}}
\end{smallmatrix}  \right)
\tikz[overlay,remember picture]{\draw[draw, double, thick](15) to [bend right=20] (39)}$
       \\
       \hline 
       \begin{tikzcd}[column sep=0.7cm, row sep=0.3cm]
1\arrow[out=150,in=210,loop,  "", "a"',distance=0.6cm]
\arrow[out=330,in=30,loop,  "", "b"',distance=0.6cm]
\end{tikzcd}
       &
       $ab,\,\,ba$
       &
       $a,\,\,b$
       &
       $ 
\tikznode{15}{ {$R$}} \hspace{2mm} \times 
\hspace{2mm}
\tikznode{39}{ {$R$}}
\tikz[overlay,remember picture]{\draw[draw, double, thick](15) to [bend right=30] (39)}$
       \\
       \hline 
       \begin{tikzcd}[column sep=0.7cm, row sep=0.3cm]
1\arrow[out=150,in=210,loop,  "", "a"',distance=0.6cm]
\arrow[out=330,in=30,loop,  "", "b"',distance=0.6cm]
\end{tikzcd}
       &
       $ a^{2},\,\,b^{2}$
       &
       $ab,\,\,ba$
       &
       $\left( \begin{smallmatrix} 
\tikznode{15}{ {$R$}} \hspace{2mm} & 
\tikznode{16}{ {$\maxideal$}}  \\
\tikznode{38}{ {$R$}}\hspace{2mm}  & 
\tikznode{39}{ {$R$}}
\end{smallmatrix}  \right)
\tikz[overlay,remember picture]{\draw[draw, double, thick](15) to [bend right=20] (39)}$
       \\
       \hline 
       \begin{tikzcd}[column sep=0.7cm, row sep=0.3cm]
1
\arrow[rr, shift right = 1mm, swap,"a"]\arrow[rr, shift right = 0.5mm, bend right=35,looseness=1, swap, "b"]
&&
2\arrow[ll,"c", shift right = 1mm, swap]
\arrow[ll, shift right = 0.5mm, bend right=35,looseness=1, swap, "d"]
\end{tikzcd}
       &
      $ ad,\,\,da,\,\,bc,\,\,cb$
       &
       $ac,\,\,ca,\,\,bd,\,\,db$
       &
       $\left( \begin{smallmatrix} 
\tikznode{15}{ {$R$}} \hspace{2mm} & 
\tikznode{16}{ {$\maxideal$}} \\
\tikznode{38}{ {$R$}}\hspace{2mm}  & 
\tikznode{39}{ {$R$}}
\end{smallmatrix}  \right)
\hspace{2mm}
\times
\hspace{2mm}
\left( \begin{smallmatrix} 
\tikznode{17}{ {$R$}} \hspace{2mm} & 
\tikznode{18}{ {$\maxideal$}}  \\
\tikznode{40}{ {$R$}}\hspace{2mm}  & 
\tikznode{41}{ {$R$}}
\end{smallmatrix}  \right)
\tikz[overlay,remember picture]{\draw[draw, double, thick](15) to [bend left=20] (17);\draw[draw, double, thick](41) to [bend left=20] (39)}$
       \\
       \hline 
       \begin{tikzcd}[column sep=0.7cm, row sep=0.3cm]
1\arrow[out=150,in=210,loop,  "", "a"',distance=0.6cm]
\arrow[r, bend right, swap,"b"]
&
2\arrow[l,"c", swap, bend right]
\arrow[out=330,in=30,loop,  "", "d"',distance=0.6cm]
\end{tikzcd}
       &
       $ac,\,\,ba,\,\,db,\,\,cd$
       &
       $a,\,\,bc,\,\,cb,\,\,d$
       &
       $\tikznode{15}{ {$R$}} 
\hspace{2mm} \times \hspace{2mm} 
\tikznode{39}{ {$R$}}
\hspace{2mm}
\times
\hspace{2mm}
\left( \begin{smallmatrix} 
\tikznode{17}{ {$R$}} \hspace{2mm} & 
\tikznode{18}{ {$\maxideal$}} \\
\tikznode{40}{ {$R$}}\hspace{2mm}  & 
\tikznode{41}{ {$R$}}
\end{smallmatrix}  \right)
\tikz[overlay,remember picture]{\draw[draw, double, thick](15) to [bend left=20] (17);\draw[draw, double, thick](41) to [bend left=20] (39)}$
       \\
       \hline 
    \end{tabular}
    \label{tab:my_label}
\end{table}

\subsubsection{Principal block of the Mathieu $11$-group}
\label{Mathieu}
\cite{Roggenkamp-rep-theory-blocks} 
In \S\ref{Mathieu} we define the $(Q,Z)$ by
\[
\begin{array}{cc}
Q=
\begin{tikzcd}[column sep=0.7cm, row sep=0.1cm]
&
&
5\arrow[dl,"j", swap]\arrow[out=60,in=120,loop,  "", "i"',distance=0.6cm]
&
\\
1\arrow[out=150,in=210,loop,  "", "a"',distance=0.6cm]
\arrow[r, bend right, swap,"b"]
&
2\arrow[l,"c", swap, bend right]\arrow[dr,"d", swap]
&
&
4\arrow[ul,"h", swap]\arrow[out=330,in=30,loop,  "", "g"',distance=0.6cm]
\\
&
&
3\arrow[ur,"f", swap]\arrow[out=240,in=300,loop,  "", "e"',distance=0.6cm]
&
\end{tikzcd}\,
&
\begin{array}{c}
 Z=\,\{ac,\,ba,\,db,\,ed,\,fe,\,gf,\,hg,\,ih,\,ji\},
\end{array}
\end{array}
\]
The $Z$-primitive cycles here are $a$, $cb$, $jhfd$, $bc$, $e$, $djhf$, $g$, $fdjh$, $i$ and $hfdj$. 

Taking $I$ as in \Cref{thm-examples-of-string-algebras}, the string algebra $\Lambda=RQ/I$ has the presentation
\[
\tikznode{15}{ {$R$}} 
\hspace{2mm} 
\times
\hspace{2mm}
\left( \begin{smallmatrix} 
\tikznode{17}{ {$R$}} \hspace{2mm} & 
\tikznode{18}{ {$\maxideal$}} \\
\tikznode{40}{ {$R$}}\hspace{2mm}  & 
\tikznode{41}{ {$R$}}
\end{smallmatrix}  \right)
\hspace{2mm} 
\times
\hspace{2mm}
\left( \begin{smallmatrix} 
\tikznode{39}{ {$R$}} \hspace{2mm} & 
\tikznode{116}{ {$\maxideal$}} \hspace{2mm} & 
\tikznode{117}{ {$\maxideal$}} \hspace{2mm}
& 
\tikznode{118}{ {$\maxideal$}} \\
\tikznode{127}{ {$R$}}\hspace{2mm}  &
\tikznode{128}{ {$R$}}\hspace{2mm}  & 
\tikznode{129}{ {$\maxideal$}} \hspace{2mm}
& 
\tikznode{130}{ {$\maxideal$}} \\
\tikznode{137}{ {$R$}}\hspace{2mm}  &
\tikznode{138}{ {$R$}}\hspace{2mm}  & 
\tikznode{139}{ {$R$}} \hspace{2mm} & 
\tikznode{140}{ {$\maxideal$}} \\
\tikznode{141}{ {$R$}}\hspace{2mm}  &
\tikznode{142}{ {$R$}}\hspace{2mm}  & 
\tikznode{143}{ {$R$}} \hspace{2mm} & 
\tikznode{144}{ {$R$}}
\end{smallmatrix}  \right)
\hspace{2mm} 
\times
\hspace{2mm}
\tikznode{2}{ {$R$}} 
\hspace{2mm} 
\times
\hspace{2mm}
\tikznode{1}{ {$R$}} 
\hspace{2mm} 
\times
\hspace{2mm}
\tikznode{3}{ {$R$}} 
\tikz[overlay,remember picture]{\draw[draw, double, thick](15) to [bend left=20] (17);\draw[draw, double, thick](41) to [bend left=20] (39);\draw[draw, double, thick](128) to [bend left=10] (1);\draw[draw, double, thick](139) to [bend right=25] (2);\draw[draw, double, thick](144) to [bend right=10] (3)}
\]
We now explain how this algebra is related to a ring  studied by Roggenkamp \cite{Roggenkamp-rep-theory-blocks}, which is Morita equivalent to the principal block of the Mathieu 11-group. 
So far we have assumed $R$ is a discrete valuation ring. From now on in \Cref{Mathieu} we shall be more specific, by letting $R=\compactBZhat_{11}$, the  $11$-adic integers, which is the ring of integers of the field $\compactBQhat_{11}$ of $11$-adic numbers.

Note that given $\zeta\in\mathbb{C}$ is a primitive $11^{\text{th}}$-root of unity we have a group isomorphism $\mathrm{Gal}(\compactBQhat_{11}(\zeta))\cong (\mathbb{Z}/11\mathbb{Z})^{\times}$ between the Galois group and the multiplicative group of integers modulo $11$, which is isomorphic to the cyclic group of order $10$. 
Hence, where $C_{5}$ denote the cyclic group of order $5$, there is an  injective group homomorphism $\alpha\colon C_{5}\to \mathrm{Gal}(\compactBQhat_{11}(\zeta))$ defining the subring $S=\mathrm{Fix}_{G}(\compactBZhat_{11}[\zeta])$ of $\compactBZhat_{11}[\zeta]$ consisting of elements fixed by the restriction of automorphisms in $\im(\alpha)$. 

In summary, by letting $\omega=\zeta+\zeta^{3}+\zeta^{4}+\zeta^{5}+\zeta^{9}\in \compactBZhat_{11}[\zeta]$ we have that $\omega^{2}-\omega+1=0$ since $\sum_{i=0}^{10}\zeta^{i}=0$. 
Moreover, we have that any $s\in S$ has the form $r+r'\omega$ where $r,r'\in\compactBZhat_{11}$. 
There is an embedding of $\Lambda$ into the $R$-algebra $B$ described in \cite[Example~I,~p.~524]{Roggenkamp-rep-theory-blocks} induced by the inclusion $R\subseteq S$. 
Indeed, the presentation of $B$ by Roggenkamp is just as we have presented $\Lambda$, but where the right-most copy of $R$ is replaced with $S$. 

\subsection{Non-trivial nerve partitions}

We now consider more examples that are outside the scope of \Cref{cor-rickes-examples} and of \Cref{subsec-examples-dvr}. 
Although we could work more generally, we stick to a concrete situation, where $R=\compactBZhat_{p}[\vert t\vert ]$, the Iwasawa algebra of power series in the variable $t$ with coefficients from $\compactBZhat_{p}$, for a fixed prime integer $p>0$. 
We now apply \Cref{thm-examples-of-string-algebras} for this choice of $R$ in two different examples. 

Let $s_{1}=p$ and $s_{2}=t$, noting that $\maxideal$ is generated by the regular sequence $(p,t)$.

\subsubsection{The running example from \S\ref{sec-primitive-cyces-and-nerve-partitions} and \S\ref{subsec-prim-reg-rad-collections}. }

Recall the details from \Cref{example-running-zeroth-non-trivial}, \Cref{example-running-first-non-trivial}
 and \Cref{example-running-second-non-trivial}. 
Here the ideal $I$ from \Cref{thm-examples-of-string-algebras} is generated by the union of $Z$ and the relations
\[
\begin{array}{c}
 \begin{array}{ccccc}
pe_{1}-x-a_{5}b_{1}, &
pe_{2}-b_{1}a_{5}, &
te_{4}-a_{2}b_{2}, &
te_{5}-b_{2}a_{2}-a_{1}zb_{3}, &
te_{6}-zb_{3}a_{1}-b_{3}a_{1}z,
\end{array}
\\
\begin{array}{ccccccc}
te_{1},\hspace{9mm} &
te_{2},\hspace{9mm} &
te_{3},\hspace{9mm} &
pe_{3},\hspace{9mm} &
pe_{4},\hspace{9mm} &
pe_{5},\hspace{9mm} &
pe_{6}. 
\end{array}
\end{array}
\]
The quotient $\Lambda=RQ/I$ can be seen to be isomorphic to the $R$-algebra
\[
\tikznode{1}{ {$\compactBZhat_{p}$}} 
\hspace{1mm}
\times
\hspace{1mm}
\left( \begin{smallmatrix} 
\tikznode{15}{ {$\compactBZhat_{p}$}} \hspace{2mm} & 
\tikznode{16}{ {p$\compactBZhat_{p}$}} \\
\tikznode{28}{ {$\compactBZhat_{p}$}}\hspace{2mm}  & 
\tikznode{29}{ {$\compactBZhat_{p}$}} 
\end{smallmatrix}  \right) 
\hspace{1mm}
\times
\hspace{1mm}
\left( \begin{smallmatrix} 
\tikznode{43}{ {$\mathbb{F}_{p}$}}\hspace{2mm}  & 
\tikznode{44}{ {$\mathbb{F}_{p}$}}  \\
\tikznode{56}{ {$0$}}\hspace{2mm}  & 
\tikznode{57}{ {$\mathbb{F}_{p}$} }
\end{smallmatrix}  \right) 
\hspace{1mm}
\times
\hspace{1mm}
\tikznode{71}{ {$\cfrac{\mathbb{F}_{p}[t]}{(t^{\ell})}$}}
\hspace{1mm}
\times
\hspace{1mm}
\left( \begin{smallmatrix} 
\tikznode{85}{ {$\mathbb{F}_{p}$}}\hspace{1mm}  & 
\tikznode{86}{ {$\mathbb{F}_{p}$}}  \\
\tikznode{98}{ {$0$}}\hspace{1mm}  &
\tikznode{99}{ {$\mathbb{F}_{p}$}}
\end{smallmatrix}  \right)
\hspace{1mm}
\times
\hspace{1mm}
\left( \begin{smallmatrix} 
\tikznode{113}{ {$\mathbb{F}_{p}[\vert t\vert ]$}}\hspace{1mm}  & 
\tikznode{114}{ {$t\mathbb{F}_{p}[\vert t\vert ]$}}   \\
\tikznode{126}{ {$\mathbb{F}_{p}[\vert t\vert ]$}}\hspace{1mm}  & 
\tikznode{127}{ {$\mathbb{F}_{p}[\vert t\vert ]$}}
\end{smallmatrix}  \right) 
\hspace{1mm}
\times
\hspace{1mm}
\left( \begin{smallmatrix} 
\tikznode{141}{ {$\mathbb{F}_{p}[\vert t\vert ]$}}\hspace{1mm} & 
\tikznode{142}{ {$t\mathbb{F}_{p}[\vert t\vert ]$}}\hspace{1mm} & 
\tikznode{143}{ {$t\mathbb{F}_{p}[\vert t\vert ]$}} \\
\tikznode{154}{ {$\mathbb{F}_{p}[\vert t\vert ]$}}\hspace{1mm} & 
\tikznode{155}{ {$\mathbb{F}_{p}[\vert t\vert ]$}}\hspace{1mm} & 
\tikznode{156}{ {$t\mathbb{F}_{p}[\vert t\vert ]$}} \\
\tikznode{167}{ {$\mathbb{F}_{p}[\vert t\vert ]$}}\hspace{1mm} & 
\tikznode{168}{ {$\mathbb{F}_{p}[\vert t\vert ]$}}\hspace{1mm} & 
\tikznode{169}{ {$\mathbb{F}_{p}[\vert t\vert ]$}} 
\end{smallmatrix}  \right) 
\tikz[overlay,remember picture]{\draw[draw, double, thick](1) to [bend left=25] (15);
\draw[draw, double, thick](29) to [bend left=20] (43);
\draw[draw, double, thick](57) to [bend left=20] (71);
\draw[draw, double, thick](71) to [bend left=20] (85);
\draw[draw, double, thick](99) to [bend right=15] (113);
\draw[draw, double, thick](127) to [bend left=15] (141);\draw[draw, double, thick](155) to (169)}
\]

\subsubsection{The example from the introduction}
\label{example-intro-example-revisited}
   Revisiting \Cref{example-introduction} and \Cref{exampple-introduction-again} from the introduction, here the string algebra over $R$ given by \Cref{thm-examples-of-string-algebras} has the presentation  
\[
\left( \begin{smallmatrix} 
\tikznode{15}{ {$\compactBZhat_{p}$}} \hspace{2mm} & 
\tikznode{16}{ {$p\compactBZhat_{p}$}} \hspace{2mm} & 
\tikznode{17}{ {$p\compactBZhat_{p}$}} \\
\tikznode{27}{ {$\compactBZhat_{p}$}}\hspace{2mm}  &
\tikznode{28}{ {$\compactBZhat_{p}$}}\hspace{2mm}  & 
\tikznode{29}{ {$p\compactBZhat_{p}$}} \\
\tikznode{37}{ {$\compactBZhat_{p}$}}\hspace{2mm}  &
\tikznode{38}{ {$\compactBZhat_{p}$}}\hspace{2mm}  & 
\tikznode{39}{ {$\compactBZhat_{p}$}}
\end{smallmatrix}  \right)
\hspace{2mm}
\times
\hspace{2mm}
\left( \begin{smallmatrix} 
\tikznode{43}{ {$\compactBZhat_{p}$}}\hspace{2mm}  & 
\tikznode{44}{ {$p\compactBZhat_{p}$}}  \\
\tikznode{56}{ {$\compactBZhat_{p}$}}\hspace{2mm}  & 
\tikznode{57}{ {$\compactBZhat_{p}$} }
\end{smallmatrix}  \right)
\hspace{2mm}
\times
\hspace{2mm}
\left( \begin{smallmatrix} 
\tikznode{85}{ {$\mathbb{F}_{p}$}}\hspace{2mm}  & 
\tikznode{86}{ {$\mathbb{F}_{p}$}}\hspace{2mm}  & 
\tikznode{87}{ {$\mathbb{F}_{p}$}}  \\
\tikznode{97}{ {$0$}}\hspace{2mm}  &
\tikznode{98}{ {$\mathbb{F}_{p}$}}\hspace{2mm}  &
\tikznode{99}{ {$\mathbb{F}_{p}$}} \\
\tikznode{107}{ {$0$}}\hspace{2mm}  &
\tikznode{108}{ {$0$}}\hspace{2mm}  &
\tikznode{109}{ {$\mathbb{F}_{p}$}} \\
\end{smallmatrix}  \right)
\hspace{2mm}
\times
\hspace{2mm}
 \begin{smallmatrix} 
\tikznode{71}{ {$\cfrac{\mathbb{F}_{p}[t]}{(t^{\ell})}$}}  
\end{smallmatrix} 
\hspace{2mm}
\times
\hspace{2mm}
 \begin{smallmatrix} 
\tikznode{72}{ {$\mathbb{F}_{p}[\vert s\vert ]$}}  
\end{smallmatrix}   
\tikz[overlay,remember picture]{\draw[draw, double, thick](15) to [bend right=20] (39);\draw[draw, double, thick](39) to [bend right=20]  (43);\draw[draw, double, thick](57) to [bend right=20]  (85);\draw[draw, double, thick](98) to [bend left=20] (72);\draw[draw, double, thick](109) to [bend left=30] (71)}
\]

\begin{acknowledgements}
The author gratefully acknowledges generous support by the Danish National Research Foundation (DNRF156); 
 the Independent Research Fund Denmark (1026-00050B); 
 and the Aarhus University Research Foundation (AUFF-F-2020-7-16). 
\end{acknowledgements}

\bibliographystyle{abbrv}
\bibliography{biblio}

\end{document}